\documentclass[intlimits,12pt,reqno]{amsart}
\usepackage{latexsym,amsthm,amsmath,amssymb,mathrsfs,mathtools}
\usepackage[T1]{fontenc}
\usepackage[utf8]{inputenc}
\usepackage[mathscr]{eucal}
\usepackage{enumerate}

\usepackage{tikz}
\usetikzlibrary{arrows,calc,patterns,cd}
\usepackage{wrapfig}

\usepackage{hyperref}
\usepackage{xcolor}
\hypersetup{
    colorlinks=true,
    linkcolor=blue,
    citecolor=blue,
    urlcolor=blue
}

\def\mvint_#1{\mathchoice
          {\mathop{\vrule width 6pt height 3 pt depth -2.5pt
                  \kern -9pt \intop}\limits_{\kern -3pt #1}}%
          {\mathop{\vrule width 5pt height 3 pt depth -2.6pt
                  \kern -6pt \intop}\nolimits_{#1}}%
          {\mathop{\vrule width 5pt height 3 pt depth -2.6pt
                  \kern -6pt \intop}\nolimits_{#1}}%
          {\mathop{\vrule width 5pt height 3 pt depth -2.6pt
                  \kern -6pt \intop}\nolimits_{#1}}}

\def\vi{\varphi}

\newcommand{\LL}{\mathcal L}

\newcommand{\bbbr}{\mathbb R}

\newcommand{\bbbn}{\mathbb N}
\newcommand{\bbbz}{\mathbb Z}

\newcommand{\overbar}[1]{\mkern 1.7mu\overline{\mkern-1.7mu#1\mkern-1.5mu}\mkern 1.5mu}

\def\diam{\operatorname{diam}}

\newtheorem{theorem}{Theorem}
\newtheorem*{theorem*}{Theorem}
\newtheorem{lemma}[theorem]{Lemma}
\newtheorem{corollary}[theorem]{Corollary}

\theoremstyle{definition}
\newtheorem{construction}[theorem]{Construction}
\newtheorem{remark}[theorem]{Remark}
\newtheorem*{remark*}{Remark}

\usepackage{setspace}

\title[Sobolev embedding]{Sobolev embedding for $M^{1,p}$ spaces is equivalent to a lower bound of the measure.}
\author[Alvarado]{Ryan Alvarado}
\address{Ryan Alvarado,\newline \indent Department of Mathematics and Statistics, Amherst College, 
\newline \indent 303 Seeley Mudd, Amherst,
Massachusetts 01002}
\email{rjalvarado\@@amherst.edu}

\author[G\'orka]{Przemys\l{}aw G\'orka}
\address{Przemys\l{}aw G\'orka,\newline \indent Department of Mathematics and Information Sciences, 
\newline \indent
Warsaw University of Technology,
\newline \indent Ul. Koszykowa 75, 00-662 Warsaw, Poland.
}
\email{pgorka\@@mini.pw.edu.pl}

\author[Haj\l{}asz]{Piotr Haj\l{}asz}
\address{Piotr Haj\l{}asz,\newline \indent Department of Mathematics, University of Pittsburgh, \newline \indent 301 Thackeray Hall, Pittsburgh,
Pennsylvania 15260}
\email{hajlasz@pitt.edu}
\thanks{P.H. was supported by NSF grant  DMS-1800457.}


\begin{document}
\maketitle
\sloppy

\begin{abstract}
It has been known since 1996 that a lower bound for the measure, $\mu(B(x,r))\geq br^s$, implies Sobolev embedding theorems for Sobolev spaces $M^{1,p}$ defined on metric-measure spaces. We prove that,  in fact Sobolev embeddings for $M^{1,p}$ spaces are equivalent to the lower bound of the measure.
\end{abstract}

\section{Introduction}

A metric-measure space $(X,d,\mu)$ is a metric space
$(X,d)$ with a Borel measure $\mu$ such that $0<\mu\big(B(x,r)\big)<\infty$ for all $x\in X$ and all $r\in(0,\infty)$.
We will always assume that metric spaces have at least two points.
Sobolev spaces on metric-measure spaces, denoted by $M^{1,p}$, have been introduced in \cite{hajlasz2}, and they play an important role in the so called area of analysis 
on metric spaces \cite{AT,BB,HeinonenK,HKST}. Later, many other definitions have been introduced in \cite{cheeger,SMP,SMP2,shanmugalingam}, but in the important 
case when the underlying metric-measure space supports the Poincar\'e inequality,
all the definitions are equivalent \cite{hajlasz,keithz}.
One of the features of the theory of $M^{1,p}$ spaces is that, unlike most of other approaches, they do not require the underlying measure to be doubling
in order to have a rich theory. In this paper we will focus on understanding the relation between the Sobolev embedding theorems for spaces $M^{1,p}$ and the growth 
properties of the measure $\mu$. 

Let $(X,d,\mu)$ be a metric-measure space.
We say that $u\in M^{1,p}(X,d,\mu)$, $0<p<\infty$, if $u\in L^p(X,\mu)$ and there is a non-negative function $0\leq g\in L^p(X,\mu)$ such that
\begin{equation}
\label{eq21}
|u(x)-u(y)|\leq d(x,y)(g(x)+g(y))
\quad
\text{for $\mu$-almost all $x,y\in X$.}
\end{equation}
More precisely, there is a set $N\subseteq X$ of measure zero, $\mu(N)=0$, such that inequality \eqref{eq21} holds for all $x,y\in X\setminus N$.
By $D(u)$ we denote the class of all functions $0\leq g\in L^p(\mu)$ for which the above inequality is satisfied, and we set
$$
D_+(u):=\big\{g\in D(u):\Vert g\Vert_{L^p(X,\mu)}>0\big\}.
$$

The space $M^{1,p}$ is equipped with a `norm'
$$
\Vert u\Vert_{M^{1,p}(X,d,\mu)}=\Vert u\Vert_{L^p(X,\mu)}+\inf_{g\in D(u)} \Vert g\Vert_{L^p(X,\mu)}.
$$
We put the word `norm' in inverted commas, because it is a norm only when $p\geq 1$. In fact, if $p\geq 1$, the space $M^{1,p}$ is a Banach space, \cite{hajlasz2}.

If $\Omega\subseteq X$ is an open set, then $(\Omega,d,\mu)$ is a metric measure space and hence, the space $M^{1,p}(\Omega,d,\mu)$ is well defined. In other words,
$u\in M^{1,p}(\Omega,d,\mu)$ if $u\in L^p(\Omega,\mu)$ and there is $0\leq g\in L^p(\Omega,\mu)$ such that \eqref{eq21} holds for almost all $x,y\in\Omega$.

The space $M^{1,p}$ is a natural generalization of the classical Sobolev space, $W^{1,p}$, because if $p>1$ and $\Omega\subseteq\bbbr^n$ is a bounded domain with the 
$W^{1,p}$-extension property, then 
$$
M^{1,p}(\Omega,d_{\bbbr^n},\LL^n)=W^{1,p}(\Omega) 
\quad
\text{and the norms are equivalent,} 
$$
see \cite{hajlasz}. 
Here we regard $\Omega$ as a metric-measure space with the Euclidean metric $d_{\bbbr^n}$, and the Lebesgue measure $\LL^n$. 
When $p=1$, the space $M^{1,1}$
in the Euclidean setting is equivalent to the Hardy-Sobolev space \cite{koskelas}. While the spaces $M^{1,p}$ for $0<p<1$ do not have an obvious
interpretation in terms of classical Sobolev spaces, they found applications to Hardy-Sobolev spaces as well as Besov spaces and Triebel-Lizorkin spaces (see, e.g., \cite{koskelas,KosYZ}). 

The classical Sobolev embedding theorems for $W^{1,p}(\bbbr^n)$ have different character when $p<n$, $p=n$, or $p>n$. Therefore, in the
metric-measure context, in order to prove embedding theorems, we need a condition that would be the counterpart of the dimension of the space. 
It turns out that such a condition is provided by the lower bound for the growth of the measure
\begin{equation}
\label{eq22}
\mu(B(x,r))\geq b r^s.
\end{equation}
With this condition one can prove Sobolev embedding theorems for $M^{1,p}$ spaces where the embedding has a different character if $0<p<s$, $p=s$ or $p>s$, \cite{hajlasz,hajlasz2}.
For a precise statement see Theorem~\ref{embedding}, below. The purpose of this paper is to prove that condition \eqref{eq22} is actually equivalent 
to the existence of the embeddings listed in Theorem~\ref{embedding}. Precise statements are given in Theorem~\ref{LMeasINT}. Partial or related results
have been obtained in \cite{carron,gorka,hajlaszkt1,hajlaszkt2,hebey,Karak1,Karak2,korobenko,korobenkomr}.
An extension of the results in this work to certain classes of Triebel-Lizorkin and Besov spaces is given in a forthcoming paper \cite{AY}.

The first main result of this paper highlights the fact
that the lower measure condition in \eqref{eq22} 
characterizes certain $M^{1,p}$-Sobolev embeddings. See
Theorem~\ref{PlessS}, Theorem~\ref{AS-U-X}, and
Theorem~\ref{CX-03} in the body of the paper for a 
more detailed account of the following theorem.

\begin{theorem}\label{LMeasINT}
Suppose that $(X,d,\mu)$ is a uniformly perfect\footnote{See \eqref{U-perf}, below.} measure
metric space and fix parameters $\sigma\in(1,\infty)$,
and $s\in(0,\infty)$.
Then the following statements are equivalent.

\begin{enumerate}[(a)]
\item There exists a finite constant $\kappa>0$ such that
\begin{equation}
\label{vvc-1}
\mu\big(B(x,r)\big)\geq \kappa\,r^s
\quad
\text{for every $x\in X$ and every finite $r\in\big(0,{\rm diam}(X)\big].$}
\end{equation}

\item There exist $p\in(0,s)$ and $C\in(0,\infty)$ such that
for every ball $B_0:=B(x_0,R_0)$ with $x_0\in X$ and finite $R_0\in\big(0,{\rm diam}(X)\big]$,  one has
\begin{equation}
\label{hdx-1}
\left(\, \mvint_{B_0} |u|^{p^*}\, d\mu\right)^{\!\!\!1/p^*}\!\!\!\leq
C\bigg(\frac{\mu(\sigma B_0)}{R_0^s}\bigg)^{\!\!1/p}\!
\left[R_0\left(\,\mvint_{\sigma B_0} g^p\, d\mu\right)^{\!\!\!1/p}
\!\!+\left(\,\mvint_{\sigma B_0}|u|^p\, d\mu\right)^{\!\!\!1/p}\right],
\end{equation}
whenever $u\in M^{1,p}(\sigma B_0,d,\mu)$ and $g\in D(u)$.
Here, $p^*=sp/(s-p)$.

\item There exist $p\in(0,s)$ and $C\in(0,\infty)$ such that
for every ball $B_0:=B(x_0,R_0)$ with $x_0\in X$ and finite $R_0\in\big(0,{\rm diam}(X)\big]$, one has
\begin{equation}
\label{natalia}
\inf_{\gamma\in\bbbr}\left(\, \mvint_{B_0} |u-\gamma|^{p^*}\, d\mu\right)^{\!\!\! 1/p^*}\leq
C\bigg(\frac{\mu(\sigma B_0)}{R_0^s}\bigg)^{\!\!\! 1/p}R_0\left(\,\mvint_{\sigma B_0} g^p\, d\mu\right)^{\!\!\! 1/p},
\end{equation}
whenever $u\in M^{1,p}(\sigma B_0,d,\mu)$ and $g\in D(u)$.

\item There exist constants $C_1,C_2,\gamma\in(0,\infty)$ such that
\begin{equation}
\label{hdx-3}
\mvint_{B_0} {\rm exp}\bigg(C_1\frac{|u-u_{B_0}|}{\|g\|_{L^{s}(\sigma B_0)}}\bigg)^{\!\!\gamma}\,d\mu\leq C_2,
\end{equation}
whenever $B_0\subseteq X$ is a ball with radius at
most ${\rm diam}(X)$, $u\in M^{1,s}(\sigma B_0,d,\mu)$ and  $g\in D_+(u)$.

\item There exist $p\in(s,\infty)$ and constant $C\in(0,\infty)$ such that 
\begin{equation}\label{hdx-4}
|u(x)-u(y)|\leq C\,d(x,y)^{1-s/p}\|g\|_{L^{p}(X,\mu)},
\qquad\forall\,x,y\in X,
\end{equation}
Hence, every function $u\in M^{1,p}(X,d,\mu)$
has H\"older continuous representative of order $(1-s/p)$
on $X$.
\end{enumerate}
\end{theorem}
Here and in what follows 
the integral average is denoted by
$$
u_E=\mvint_Eu\, d\mu =\frac{1}{\mu(E)}\int_E u\, d\mu
$$
where $E\subseteq X$ is a $\mu$-measurable set of positive measure. Also,
$\tau B$ denotes the dilation of a ball $B$ by a factor $\tau\in(0,\infty)$, i.e., $\tau B:=B(x,\tau r)$.

\begin{remark}
The expression {\em `for every finite $r\in (0,\diam(X)]$'} is a concise way of saying {\em `for every $r\in (0,\diam(X)]$ if $\diam(X)<\infty$ and for every $r\in (0,\infty)$ if $\diam(X)=\infty$'.}
\end{remark}

\begin{remark}
Theorem~\ref{LMeasINT} asserts that in particular, if just {\em one} of the Sobolev embeddings {\em (b)}\,-{\em (e)} holds for {\em some} $p$, then all of the embeddings {\em (b)}\,-{\em (e)} hold for all $p$. This is a new self-improvement phenomena.  
\end{remark}

Given a metric-measure space, $(X,d,\mu)$, the measure $\mu$ is said to be {\tt doubling} provided there
exists a constant $C\in(0,\infty)$ such that
\begin{equation}\label{doub}
\mu(2B)\leq C\mu(B)\,\,\mbox{ for all balls\,\,$B\subseteq X$.}
\end{equation}
The smallest constant $C$ for which  \eqref{doub} is satisfied will be denoted
by $C_\mu$. It follows from \eqref{doub} that if $X$ contains at least two elements then 
$C_\mu>1$ (see \cite[Proposition~3.1, p.\,72]{AlMi}).
Moreover, as is well-known, the doubling property in \eqref{doub} implies the following quantitative
condition: for each $s\in\big[\log_2(C_\mu),\infty\big)$,
there exists $\kappa\in(0,\infty)$ satisfying
\begin{equation}
\label{Doub-2}
\frac{\mu\big(B(x,r)\big)}{\mu\big(B(y,R)\big)}\geq \kappa\bigg(\frac{r}{R}\bigg)^{\!\!s},
\end{equation}
whenever $x,y\in X$ satisfy $B(x,r)\subseteq B(y,R)$ and $0<r\leq R<\infty$ (see, e.g., \cite[Lemma~4.7]{hajlasz}).
Conversely, any measure satisfying \eqref{Doub-2} for
some $s\in(0,\infty)$ is necessarily doubling. 
Note that if the space $X$ is bounded then the above quantitative doubling property implies the lower measure bound in \eqref{vvc-1}.

The following theorem, which constitutes the second main result of our paper, is an analogue of Theorem~\ref{LMeasINT} for doubling measures. The reader is referred to Theorem~\ref{EMT}
and Theorem~\ref{CX-032}.

\begin{theorem}\label{DoubMeasINT}
Suppose that $(X,d,\mu)$ is a uniformly perfect measure
metric space and fix parameters $\sigma\in(1,\infty)$,
and $s\in(0,\infty)$.
Then the following statements are equivalent.
\begin{enumerate}[(a)]
\item There exists a constant $\kappa\in(0,\infty)$ satisfying
\begin{equation*}
\frac{\mu\big(B(x,r)\big)}{\mu\big(B(y,R)\big)}\geq \kappa\bigg(\frac{r}{R}\bigg)^{\!\!s},
\end{equation*}
whenever $x,y\in X$ satisfy $B(x,r)\subseteq B(y,R)$ and $0<r\leq R<\infty$.

\item There exist $p\in(0,s)$ and $C\in(0,\infty)$ such that
for every ball $B_0:=B(x_0,R_0)$ with $x_0\in X$ and $R_0\in(0,\infty)$, one has
\begin{equation*}
\left(\, \mvint_{B_0} |u|^{p^*}\, d\mu\right)^{\!\!\!1/p^*}\leq
CR_0\left(\, \mvint_{\sigma B_0} g^p\, d\mu\right)^{\!\!\!1/p} 
+ C\left(\, \mvint_{\sigma B_0}|u|^p\, d\mu\right)^{\!\!\!1/p},
\end{equation*}
whenever $u\in M^{1,p}(\sigma B_0,d,\mu)$ and $g\in D(u)$.

\item There exist $p\in(0,s)$ and $C\in(0,\infty)$ such that
for every ball $B_0:=B(x_0,R_0)$ with $x_0\in X$ and $R_0\in(0,\infty)$, one has
\begin{equation*}
\inf_{\gamma\in\bbbr}\left(\, \mvint_{B_0} |u-\gamma|^{p^*}\, d\mu\right)^{\!\!\!1/p^*}\leq
C R_0\left(\, \mvint_{\sigma B_0} g^p\, d\mu\right)^{\!\!\!1/p},
\end{equation*}
whenever $u\in M^{1,p}(\sigma B_0,d,\mu)$ and $g\in D(u)$.

\item There exist $p\in(s,\infty)$ and $C\in(0,\infty)$ 
with the property that for each $u\in M^{1,p}(X,d,\mu)$ and $g\in D(u)$, and each ball $B_0:=B(x_0,R_0)$ with $x_0\in X$ and $R_0\in\big(0,{\rm diam}(X)\big]$, finite, there holds
\begin{equation*}
|u(x)-u(y)|\leq C\,d(x,y)^{1-s/p}R_0^{s/p}\left(\,\mvint_{\sigma B_0}g^p\,d\mu\right)^{\!\!\!1/p}
\quad\mbox{for every } x,y\in B_0.
\end{equation*}
Hence, every function $u\in M^{1,p}(\sigma B_0,d,\mu)$
has H\"older continuous representative of order $(1-s/p)$
on $B_0$.
\end{enumerate}
\end{theorem}

\begin{remark}
Note that Theorem~\ref{DoubMeasINT} does not cover the case $p=s$. For that case see Theorems~\ref{AS-U-BDD} and~\ref{AS-U-BDD2} in the body of the paper.
\end{remark}

\subsection{Notation}

Open (metric) balls in a given metric space, $(X,d)$, shall be denoted by $B(x,r)=\{y:\, d(x,y)<r\}$ while the notation $\overbar{B}(x,r)=\{y\in X:\, d(x,y)\leq r\}$ will be used 
for closed balls. We allow the radius of a ball to  equal zero. If $r=0$, then
$B(x,r)=\emptyset$, but $\overbar{B}(x,r)=\{ x\}$. 
As a sign or warning, note that in general $\overbar{B}(x,r)$ is not necessarily equal to the topological closure of $B(x,r)$. By $C$ we will denote a general constant whose value may change within a single string of estimates.
While the center and the radius of a ball in a metric space is not necessarily uniquely determined, our balls will always have specified centers and radii so formally a ball is a triplet:\ a set, a center and a radius.
By writing $C(s,p)$ we will mean that the constant depends on parameters $s$ and $p$ only. 
$\mathbb{N}$ will stand for the set of positive integers, while $\mathbb{N}_0:=\mathbb{N}\cup\{ 0\}$. The characteristic function of a set $E$ will be denoted by $\chi_E$.

\section{Sobolev embedding on metric-measure spaces}
\label{S2}

The next result from \cite[Theorem~8.7]{hajlasz} provides a general embedding theorem for Sobolev spaces $M^{1,p}$ defined on balls in 
a metric measure space $X$. While this result has been proven in \cite{hajlasz} we decided to include a proof for the following reasons.
The paper \cite{hajlasz} does not include inequality \eqref{eq18}. While in the case $p^*\geq 1$, inequality \eqref{eq18} easily follows from \eqref{eq19}
(proven in \cite{hajlasz}) we do not know how to conclude it from \eqref{eq19} when $p^*<1$. Also some of the arguments given in \cite{hajlasz} are somewhat sketchy
and hard to follow so we decided that the result needs a complete and a detailed proof. At last, but not least, this result plays a fundamental role in the current paper
and proving it here makes the paper more complete and easier to comprehend. 
The proof presented below is similar, but slightly different than that in \cite{hajlasz}.
To facilitate the formulation of 
the result we introduce the following piece of notation.
Given constants $s,b\in(0,\infty)$, $\sigma\in[1,\infty)$  and a ball $B_0\subseteq X$ of radius $R_0$, the measure $\mu$ is said to satisfy the $V(\sigma B_0,s,b)$ condition\footnote{This condition
is a slight variation of the one in \cite[p.\,197]{hajlasz}.} provided
\begin{equation}
\label{measbound}
\mu\big(B(x,r)\big)\geq br^s
\quad
\text{whenever} 
\quad
\text{$B(x,r)\subseteq\sigma B_0$ and $r\in(0,\sigma R_0]$.}
\end{equation}
In this section we will only consider balls $B(x,r)$ with $r\in(0,\infty)$.

\begin{theorem}
\label{embedding}
Let $u\in M^{1,p}(\sigma B_0,d,\mu)$ and $g\in D(u)$, where $0<p<\infty$, $\sigma>1$ and
$B_0$ is a ball of radius $R_0$. Assume that the measure $\mu$ satisfies the condition $V(\sigma B_0,s,b)$.
Then there exist constants $C$, $C_1$ and $C_2$ depending on $s$, $p$ and $\sigma$ only such that
\begin{enumerate}
\item[(a)]
If $0<p<s$, then $u\in L^{p^*}(B_0)$, where $p^*=sp/(s-p)$, and the following inequalities are satisfied,
\begin{equation}
\label{eq18}
\left(\, \mvint_{B_0} |u|^{p^*}\, d\mu\right)^{\!\!\!1/p^*}\leq
C\left(\frac{\mu(\sigma B_0)}{bR_0^s}\right)^{\!\!\!1/p}\, R_0\left(\, \mvint_{\sigma B_0} g^p\, d\mu\right)^{\!\!\!1/p} 
+ C\left(\, \mvint_{\sigma B_0}|u|^p\, d\mu\right)^{\!\!\!1/p},
\end{equation}
\begin{equation}
\label{eq19}
\inf_{\gamma\in\bbbr}\left(\, \mvint_{B_0} |u-\gamma|^{p^*}\, d\mu\right)^{\!\!\!1/p^*}\leq
C\left(\frac{\mu(\sigma B_0)}{bR_0^s}\right)^{\!\!\!1/p}\, R_0\left(\, \mvint_{\sigma B_0} g^p\, d\mu\right)^{\!\!\!1/p}. 
\end{equation}
\item[(b)] If $p=s$ and $g\in D_+(u)$, then
\begin{equation*}
\mvint_{B_0}\exp\left( C_1b^{1/s}\frac{|u-u_{B_0}|}{\Vert g\Vert_{L^s(\sigma B_0)}}\right)\, d\mu\leq C_2. 
\end{equation*}
\item[(c)] If $p>s$, then
\begin{equation}
\label{eq29}
\Vert u-u_{B_0}\Vert_{L^\infty(B_0)}\leq 
C\left(\frac{\mu(\sigma B_0)}{bR_0^s}\right)^{\!\!\!1/p} R_0\left(\, \mvint_{\sigma B_0} g^p\, d\mu\right)^{\!\!\!1/p}.
\end{equation}
In particular, $u$ has a H\"older continuous representative on $B_0$ and
\begin{equation}
\label{eq30}
|u(x)-u(y)|\leq C b^{-1/p}d(x,y)^{1-s/p}\left(\, \int_{\sigma B_0} g^p\, d\mu\right)^{\!\!\!1/p}
\quad
\text{for all $x,y\in B_0$.}
\end{equation}
\end{enumerate}
\end{theorem}
\begin{remark}
If $p^*\geq 1$, then H\"older's inequality yields
\begin{equation}
\label{eq20}
\left(\, \mvint_{B_0} |u-u_{B_0}|^{p^*}\, d\mu\right)^{\!\!\!1/p^*}\leq
2\inf_{\gamma\in\bbbr}\left(\, \mvint_{B_0} |u-\gamma|^{p^*}\, d\mu\right)^{\!\!\!1/p^*}
\end{equation}
and hence we can replace the expression on the left hand side of \eqref{eq19} with the one
on the left hand side of \eqref{eq20}. Moreover, if in addition $p\geq1$, we also have that \eqref{eq18} easily follows from this new version of
\eqref{eq19}. However, we do not know how to conclude \eqref{eq18} from \eqref{eq19} when $p<1$.
\end{remark}
\begin{proof}
Throughout the proof by $C$ we will denote a generic constant that depends on $p$, $s$ and $\sigma$ {\em only}. The dependence 
of other quantities like $b$, $R_0$, $u$ or $g$ will be given in an explicit form. By writing $A\approx B$ we will mean that
the quantities $A$ and $B$ are non-negative and there is a constant $C\geq 1$ (depending on $p$, $s$ and $\sigma$ only)
such that $C^{-1}A\leq B\leq CA$.

Clearly, we can assume that $\int_{\sigma B_0}g^p\, d\mu>0$. Indeed, if the integral equals zero, $u$ is constant and the result is obvious. 

By replacing, if necessary, $g$ with $\widetilde{g}=g+\left(\mvint_{\sigma B_0} g^p\, d\mu\right)^{\!\!\!1/p}$ we may further assume that
\begin{equation}
\label{eq4}
g(x)\geq 2^{-(1+1/p)}\left(\, \mvint_{\sigma B_0} g^p\,  d\mu\right)^{\!\!\!1/p}>0
\quad
\text{for $\mu$-almost every $x\in\sigma B_0$.}
\end{equation}
Let $N\subseteq \sigma B_0$ be a set of measure zero such that the pointwise inequality \eqref{eq21} holds for all $x,y\in \sigma B_0\setminus N$. Define the sets
$$
E_j:=\{ x\in\sigma B_0\setminus N:\, g(x)\leq 2^j\},
\quad
j\in\bbbz.
$$
Clearly $E_{j-1}\subseteq E_j$. Since by \eqref{eq4}, $g>0$ almost everywhere in $\sigma B_0$,
\begin{equation}
\label{eq42}
\mu\left(\sigma B_0\setminus\bigcup_{j=-\infty}^\infty (E_j\setminus E_{j-1})\right)=0.
\end{equation}
It follows from the pointwise inequality \eqref{eq21} that
$u$ restricted to $E_j$ is $2^{j+1}$-Lipschitz, i.e.,
\begin{equation}
\label{eq8}
|u(x)-u(y)|\leq 2^{j+1}d(x,y)
\quad
\text{for all $x,y\in E_j$}.
\end{equation}
Also, the measure of the complement of each of the sets $E_j$ can be easily estimated from 
Chebyschev's inequality as follows,
\begin{equation}
\label{eq7}
\mu(\sigma B_0\setminus E_j) = \mu(\{x\in\sigma B_0:\, g(x)>2^j\})\leq 2^{-jp}\int_{\sigma B_0} g^p\, d\mu.
\end{equation}
Note that
\begin{equation}
\label{eq5}
\int_{\sigma B_0} g^p\, d\mu\approx
\sum_{j=-\infty}^\infty 2^{jp}\mu(E_j\setminus E_{j-1}).
\end{equation}
Fix $\gamma\in\bbbr$ arbitrarily and let
\begin{equation}
\label{eq43}
a_j:=\sup_{E_j\cap B_0} |u-\gamma|
\quad
\text{with}
\quad
a_j:=0 
\quad
\text{if} 
\quad
E_j\cap B_0=\emptyset.
\end{equation}
Clearly, $a_j\leq a_{j+1}$ and for $0<p<s$ we have
\begin{equation}
\label{eq6}
\int_{B_0} |u-\gamma|^{p^*}\, d\mu \leq\sum_{j=-\infty}^\infty a_j^{p^*}\mu(B_0\cap(E_j\setminus E_{j-1})).
\end{equation}
Note that we used here \eqref{eq42}, because we need to know that the sets $E_j\setminus E_{j-1}$ cover almost all points in the set $B_0$.

The idea of the proof in the case $0<p<s$ is to estimate the series at \eqref{eq6} by the series in \eqref{eq5}. Similar
ideas are also used in other cases $p=s$ and $p>s$.

We will need the following elementary result.
\begin{lemma}
\label{joasia}
If $B(x,r)\subseteq\sigma B_0$ and $\mu\big(B(x,r)\big)\geq 2\mu(\sigma B_0\setminus E_j)$ for some $j\in\bbbz$, then
$$
\mu(B(x,r)\cap E_j)\geq\frac{1}{2}\mu(B(x,r))>0.
$$
\end{lemma}
\begin{proof}
Observe that
\begin{equation}
\begin{split}
0<\mu\big(B(x,r)\big)
&=\mu\big(B(x,r)\cap E_j\big)+\mu\big(B(x,r)\setminus E_j\big)
\leq\mu\big(B(x,r)\cap E_j\big)+\mu\big(\sigma B_0\setminus E_j\big)\\
&\leq\mu\big(B(x,r)\cap E_j\big)+\frac{1}{2}\mu\big(B(x,r)\big).
\end{split}
\nonumber
\end{equation}
The claim now follows.
\end{proof}
Let $k_0$ be the least integer such that
\begin{equation}
\label{eq15}
2^{k_0}\geq \left(\frac{2^{1/s}}{(1-2^{-p/s})(\sigma-1)}\right)^{\!\!\!s/p}(bR_0^s)^{-1/p}\left(\, \int_{\sigma B_0} g^p\, d\mu\right)^{\!\!\!1/p}.
\end{equation}
Then
\begin{equation}
\label{eq16}
2^{k_0}\approx (bR_0^s)^{-1/p}\left(\, \int_{\sigma B_0} g^p\, d\mu\right)^{\!\!\!1/p}
\end{equation}
Note that condition \eqref{eq15} is equivalent to
\begin{equation}
\label{eq2}
2^{-k_0p/s}\frac{2^{1/s}b^{-1/s}}{1-2^{-p/s}}\,\left(\,\int_{\sigma B_0} g^p\,d\mu\right)^{\!\!\!1/s}\leq (\sigma -1)R_0.
\end{equation}

\begin{lemma}
\label{piotr}
Under the above assumptions $\mu(E_{k_0})\geq\mu(\sigma B_0)/2$.
\end{lemma}
\begin{proof}
Suppose to the contrary that $\mu(E_{k_0})<\mu(\sigma B_0)/2$. 
Then 
\begin{equation}
\label{eq41}    
\mu(\sigma B_0\setminus E_{k_0})>\mu(\sigma B_0)/2.
\end{equation}
Inequalities \eqref{eq7} and \eqref{eq2} yield
$$
r:= 2^{1/s}b^{-1/s}\mu(\sigma B_0\setminus E_{k_0})^{1/s}\leq
2^{1/s} b^{-1/s} 2^{-k_0p/s}\left(\,\int_{\sigma B_0} g^p\, d\mu\right)^{\!\!\!1/s}
<(\sigma-1)R_0.
$$
Therefore, if $z_0$ is the center of the ball $B_0$, then $B(z_0,r)\subseteq\sigma B_0$ so the $V(\sigma B_0,s,b)$ condition and \eqref{eq41} give
$$
\mu(\sigma B_0)\geq \mu(B(z_0,r))\geq br^s=2\mu(\sigma B_0\setminus E_{k_0})>\mu(\sigma B_0),
$$
which is an obvious contradiction.
\end{proof}

For $k>k_0$ and $i=0,1,\ldots,k-k_0-1$ we define
$$
r_{k-i}:=2^{1/s}b^{-1/s}2^{-(k-(i+1))p/s}\left(\,\int_{\sigma B_0}g^p\, d\mu\right)^{1/s}.
$$
Note that
\begin{equation}
\label{michal}
\begin{split}
r_k+r_{k-1}+\ldots+r_{k_0+1} 
&= 
2^{1/s}b^{-1/s}\left(\,\int_{\sigma B_0} g^p\, d\mu\right)^{1/s}\sum_{i=0}^{k-k_0-1} 2^{-(k-(i+1))p/s}\\
&<
2^{1/s}b^{-1/s}\left(\,\int_{\sigma B_0} g^p\, d\mu\right)^{1/s}\sum_{i=-\infty}^{k-k_0-1} 2^{-(k-(i+1))p/s}\\
&=
2^{-k_0 p/s}\frac{2^{1/s}b^{-1/s}}{1-2^{-p/s}}\left(\,\int_{\sigma B_0} g^p\, d\mu\right)^{1/s}\leq (\sigma-1)R_0.
\end{split}
\end{equation}
Assume that $E_k\cap B_0\neq\emptyset$ and
choose $x_k\in E_k\cap B_0$ arbitrarily. 
We will now use induction with respect to $i$ and define a sequence $x_{k-i}$, $i=1,\ldots,k-k_0$ such that $x_{k-i}\in\sigma B_0$ and
$$
x_{k-1}\in E_{k-1}\cap B(x_k,r_k),\
x_{k-2}\in E_{k-2}\cap B(x_{k-1},r_{k-1}),\ldots,
x_{k_0}\in E_{k_0}\cap B(x_{k_0+1},r_{k_0+1}).
$$
For $i=1$ we construct $x_{k-1}$ as follows. According to \eqref{michal}, $B(x_k,r_k)\subseteq \sigma B_0$ and hence the volume condition $V(\sigma B_0,s,b)$ and Chebyschev's inequality \eqref{eq7} yield
$$
\mu(B(x_k,r_k))\geq br_k^s=2\cdot 2^{-(k-1)p}\int_{\sigma B_0} g^p\, d\mu\geq2\mu(\sigma B_0\setminus E_{k-1}).
$$
Therefore, Lemma~\ref{joasia} implies that
$\mu (B(x_k,r_k)\cap E_{k-1})>0$
and we can find $x_{k-1}\in E_{k-1}\cap B(x_k,r_k)$. 
Clearly $x_{k-1}\in\sigma B_0$.
Suppose now that we already selected point
$x_{k-1},\ldots, x_{k-i}$ for some $1\leq i<k-k_0$ satisfying
\begin{equation*}
x_{k-j}\in \sigma B_0\cap E_{k-j}\cap B(x_{k-j+1},r_{k-j+1})
\quad
\text{for $j=1,\ldots,i$}.
\end{equation*}
It remains to show that we can select
$$
x_{k-(i+1)}\in \sigma B_0\cap E_{k-(i+1)}\cap B(x_{k-i},r_{k-i}).
$$
For any $y\in B(x_{k-i},r_{k-i})$ we have
\begin{equation*}
\begin{split}
d(y,x_k)&
\leq d(y,x_{k-i})+d(x_{k-i},x_{k-i+1})+\ldots+ d(x_{k-1},x_k)\\
&<
r_{k-i}+r_{k-i+1}+\ldots+r_k\leq (\sigma-1)R_0.
\end{split}
\end{equation*}
Since $x_k\in B_0$, it follows that $B(x_{k-i},r_{k-i})\subseteq\sigma B_0$.
Therefore the volume condition $V(\sigma B_0,s,b)$ and Chebyschev's inequality \eqref{eq7} yield
$$
\mu(B(x_{k-i},r_{k-i}))\geq br_{k-i}^s=2\cdot 2^{-(k-(i+1))p}\int_{\sigma B_0} g^p\, d\mu\geq2\mu(\sigma B_0\setminus E_{k-(i+1)}).
$$
Thus, Lemma~\ref{joasia} yields $\mu(B(x_{k-i},r_{k-i})\cap E_{k-(i+1)})>0$ and we can find
$$
x_{k-(i+1)}\in E_{k-(i+1)}\cap B(x_{k-i},r_{k-i}).
$$
Clearly $x_{k-(i+1)}\in\sigma B_0$. That completes the inductive argument.

Note that for $i=0,1,\ldots,k-k_0-1$,
$$
d(x_{k-i},x_{k-(i+1)})<r_{k-i}=
2^{1/s}b^{-1/s}2^{-(k-(i+1))p/s}\left(\,\int_{\sigma B_0} g^p\, d\mu\right)^{1/s}.
$$
Since $x_{k-i},x_{k-(i+1)}\in E_{k-i}$, $u$ is $2^{k-i+1}$-Lipschitz on $E_{k-i}$, and
$x_{k_0}\in E_{k_0}\cap\sigma B_0$
we have
\begin{equation}
\label{eq45}
\begin{split}
|u(x_k)-\gamma|
&\leq
\left(\sum_{i=0}^{k-k_0-1}|u(x_{k-i})-u(x_{k-(i+1)})|\right) + |u(x_{k_{0}})-\gamma|\\
&\leq
\left(\sum_{i=0}^{k-k_0-1} 2^{k-i+1} d(x_{k-i},x_{k-(i+1)})\right) + \sup_{E_{k_0}\cap\sigma B_0}|u-\gamma|\\
&<
4\cdot 2^{1/s}b^{-1/s}\left(\,\int_{\sigma B_0} g^p\, d\mu\right)^{1/s}
\sum_{i=0}^{k-k_0-1} 2^{(k-(i+1))(1-p/s)}+\sup_{E_{k_0}\cap\sigma B_0}|u-\gamma|\\
&=
4\cdot 2^{1/s}b^{-1/s}\left(\,\int_{\sigma B_0} g^p\, d\mu\right)^{1/s}
\sum_{j=k_0}^{k-1} 2^{j(1-p/s)} + \sup_{E_{k_0}\cap\sigma B_0}|u-\gamma|.
\end{split}
\end{equation}
This yields
\begin{equation}
\label{eq44}
a_k\leq 
4\cdot 2^{1/s}b^{-1/s}\left(\,\int_{\sigma B_0} g^p\, d\mu\right)^{1/s}
\sum_{j=k_0}^{k-1} 2^{j(1-p/s)} + \sup_{E_{k_0}\cap\sigma B_0}|u-\gamma|
\quad
\text{for all $k>k_0$.}
\end{equation}
Indeed, if $E_k\cap B_0\neq\emptyset$, then $a_k=\sup_{E_k\cap B_0}|u-\gamma|$.
Since $x_k\in E_k\cap B_0$ was selected {\em arbitrarily},
taking the supremum in \eqref{eq45} over $x_k\in E_k\cap B_0$ yields \eqref{eq44}. If $E_k\cap B_0=\emptyset$, then $a_k=0$ (see \eqref{eq43}) and \eqref{eq44} is trivially true.

Since $\mu(E_{k_0})>0$, by Lemma~\ref{piotr}, we can take $y\in E_{k_0}$. If $\gamma =u(y)$,  
then the Lipschitz continuity \eqref{eq8}, and \eqref{eq16} yield
\begin{equation}
\label{eq17}
\sup_{E_{k_0}\cap \sigma B_0}|u-\gamma|
\leq 2^{k_0+1}\diam(\sigma B_0) \leq
2^{k_0+2}\sigma R_0
\leq
C R_0(bR_0^s)^{-1/p}\left(\, \int_{\sigma B_0} g^p\, d\mu\right)^{\!\!\!1/p}.
\end{equation}

\noindent
{\em Proof of (a).}
First we will prove inequality \eqref{eq19}.
Let $\gamma=u(y)$ as in \eqref{eq17}. Recall that $a_k=\sup_{E_k\cap B_0}|u-\gamma|$.
Since $2^{1-p/s}>1$, we can estimate the finite sum in \eqref{eq44} by the convergent geometric series $\sum_{j=-\infty}^{k-1}2^{j(1-p/s)}$ so \eqref{eq44}
gives
$$
a_k\leq Cb^{-1/s}\left(\, \int_{\sigma B_0} g^p\, d\mu\right)^{\!\!\!1/s} 2^{k(1-p/s)}+\sup_{E_{k_0}\cap\sigma B_0} |u-\gamma|
\quad
\text{for all $k>k_0$.}
$$
However, since
$$
a_k=\sup_{E_k\cap B_0}|u-\gamma|\leq \sup_{E_{k_0}\cap\sigma B_0}|u-\gamma|
\quad
\text{for $k\leq k_0$}
$$ 
we actually have
$$
a_k\leq Cb^{-1/s}\left(\, \int_{\sigma B_0} g^p\, d\mu\right)^{\!\!\!1/s} 2^{k(1-p/s)}+\sup_{E_{k_0}\cap\sigma B_0} |u-\gamma|
\quad
\text{for all $k\in\bbbz$.}
$$
Therefore, \eqref{eq6}, \eqref{eq5}, and \eqref{eq17} yield
\begin{equation*}
\begin{split}
&\int_{B_0} |u-\gamma|^{p^*}\, d\mu\\
&\leq
C b^{-p^*/s}\left(\, \int_{\sigma B_0} g^p\, d\mu\right)^{\!\!\!p^*/s}\sum_{k=-\infty}^\infty 2^{kp}\mu(E_{k}\setminus E_{k-1}) 
+ C\left(\sup_{E_{k_0}\cap\sigma B_0}|u-\gamma|\right)^{\!\!p^*}\mu(B_0)\\
&\leq 
C b^{-p^*/s}\left(\, \int_{\sigma B_0} g^p\, d\mu\right)^{\!\!\!p^*/p}
+ C\left(\sup_{E_{k_0}\cap\sigma B_0}|u-\gamma|\right)^{\!\!p^*}\mu(B_0)\\
&\leq
Cb^{-p^*/s}\left(1+\frac{\mu(B_0)}{bR_0^s}\right)\left(\, \int_{\sigma B_0} g^p\, d\mu\right)^{\!\!\!p^*/p}
\leq
Cb^{-p^*/s}\frac{\mu(B_0)}{bR_0^s}\left(\, \int_{\sigma B_0} g^p\, d\mu\right)^{\!\!\!p^*/p}. 
\end{split}
\end{equation*}
In the last inequality we used the condition $V(\sigma B_0,s,b)$ to estimate
$1+\mu(B_0)/(bR_0^s)\leq 2\mu(B_0)/(bR_0^s)$. The above estimate easily implies inequality \eqref{eq19}.

Now it remains to prove inequality \eqref{eq18}. Take $\gamma =0$. 
Then $a_k=\sup_{E_k\cap B_0} |u|$. Let $b_{k_0}:=\inf_{E_{k_0}\cap\sigma B_0}|u|$. 
Since
$$
b_{k_0}^p\chi_{E_{k_0}}\leq |u|^p\chi_{\sigma B_0},
$$
Lemma~\ref{piotr} yields
$$
\frac{\mu(\sigma B_0)}{2} b_{k_0}^p\leq
b_{k_0}^p\mu(E_{k_0}) \leq
\int_{\sigma B_0} |u|^p\, d\mu
\qquad
\text{so}
\qquad
b_{k_0}\leq 2^{1/p}\left(\, \mvint_{\sigma B_0}|u|^p\, d\mu\right)^{\!\!\!1/p}\, .
$$
The Lipschitz continuity \eqref{eq8}, and \eqref{eq16} yield
\begin{equation*}
\begin{split}
\sup_{E_{k_0}\cap \sigma B_0}|u|
&\leq 2^{k_0+1}\diam(\sigma B_0) + b_{k_0}\leq
2^{k_0+2}\sigma R_0+ 2^{1/p}\left(\, \mvint_{\sigma B_0} |u|^p\, d\mu\right)^{\!\!\!1/p}\\
&\leq
C\left(R_0(bR_0^s)^{-1/p}\left(\, \int_{\sigma B_0} g^p\, d\mu\right)^{\!\!\!1/p} + \left(\, \mvint_{\sigma B_0} |u|^p\, d\mu\right)^{\!\!\!1/p}\right).
\end{split}
\end{equation*}
Hence, a similar calculation as above gives
\begin{equation*}
\begin{split}
\int_{B_0} |u|^{p^*}\, d\mu
&\leq
C b^{-p^*/s}\left(\, \int_{\sigma B_0} g^p\, d\mu\right)^{\!\!\!p^*/s}\sum_{k=-\infty}^\infty 2^{kp}\mu(E_{k}\setminus E_{k-1}) 
+ C\left(\sup_{E_{k_0}\cap\sigma B_0}|u|\right)^{\!\!p^*}\mu(B_0)\\
&\leq 
C b^{-p^*/s}\left(\, \int_{\sigma B_0} g^p\, d\mu\right)^{\!\!\!p^*/p}
+ C\left(\sup_{E_{k_0}\cap\sigma B_0}|u|\right)^{\!\!p^*}\mu(B_0)\\
&\leq
Cb^{-p^*/s}\left(1+\frac{\mu(B_0)}{bR_0^s}\right)\left(\, \int_{\sigma B_0} g^p\, d\mu\right)^{\!\!\!p^*/p} +
C\left(\, \mvint_{\sigma B_0} |u|^p\, d\mu\right)^{\!\!\!p^*/p}\mu(B_0)\\
&\leq
Cb^{-p^*/s}\frac{\mu(B_0)}{bR_0^s}\left(\, \int_{\sigma B_0} g^p\, d\mu\right)^{\!\!\!p^*/p} +
C\left(\, \mvint_{\sigma B_0} |u|^p\, d\mu\right)^{\!\!\!p^*/p}\mu(B_0).
\end{split}
\end{equation*}
This estimate easily imply inequality \eqref{eq18}.

For the parts {\em (b)} and {\em (c)}, observe that since $p\geq s$, we have that $u\in M^{1,q}(\sigma B_0,d,\mu)$, where $q=s/(s+1)$, and hence $u\in L^{q^*}(B_0)=L^1(B_0)$ by part {\em (a)}. Therefore, $u_{B_0}$ is well defined and finite.

\noindent
{\em Proof of (b).}
Let $\gamma=u(y)$ be as in \eqref{eq17}
and $a_k=\sup_{E_k\cap B_0} |u-\gamma|$.
For $a>0$, Jensen's inequality and convexity of $e^t$ yield
\begin{equation*}
\begin{split}
&\mvint_{B_0}e^{a|u(x)-u_{B_0}|}\, d\mu(x)
\leq
\mvint_{B_0}\exp\left(\, \mvint_{B_0}a|u(x)-u(y)|\, d\mu(y)\right)\, d\mu(x)\\
&\leq
\mvint_{B_0}\mvint_{B_0}e^{a|u(x)-u(y)|}d\mu(y)\, d\mu(x)
\leq
\mvint_{B_0}e^{a|u(x)-\gamma|}\, d\mu(x)\, 
\mvint_{B_0}e^{a|u(y)-\gamma|}\, d\mu(y) \\
&=
\left(\, \mvint_{B_0}e^{a|u(x)-\gamma|}\, d\mu(x)\right)^{\!\!2}. 
\end{split}
\end{equation*}
Hence
\begin{equation}
\label{eq23}
\mvint_{B_0}\exp\left(C_1b^{1/s}\frac{|u-u_{B_0}|}{\Vert g\Vert_{L^s(\sigma B_0)}}\right)\, d\mu
\leq
\left(\, \mvint_{B_0}\exp\left( \frac{C_1b^{1/s}|u-\gamma|}{\Vert g\Vert_{L^s(\sigma B_0)}}\right)\, d\mu\right)^{\!\!2}
\end{equation}
and thus it suffices to estimate the right hand side of \eqref{eq23}, where $C_1$ is to be chosen.

Since $p=s$, inequality \eqref{eq17} reads as
\begin{equation}
\label{eq24}
\sup_{E_{k_0}\cap\sigma B_0} |u-\gamma|\leq
Cb^{-1/s}\left(\, \int_{\sigma B_0} g^s\, d\mu\right)^{\!\!\!1/s}.
\end{equation}
Given that $2^{j(1-s/p)}=1$, \eqref{eq44} and \eqref{eq24} yield
\begin{equation}
\label{eq25}
a_k\leq 
\widetilde{C} b^{-1/s}\left(\, \int_{\sigma B_0} g^s\, d\mu\right)^{\!\!\!1/s}(k-k_0)
\quad\text{for $k>k_0$.}
\end{equation}
It follows from \eqref{eq24} that
\begin{equation}
\label{eq26}
\frac{C_1b^{1/s}|u(x)-\gamma|}{\Vert g\Vert_{L^s(\sigma B_0)}}\leq C C_1
\quad
\text{for all $x\in E_{k_0}$,}
\end{equation}
while \eqref{eq25} yields
\begin{equation}
\label{eq27}
\frac{C_1b^{1/s}|u(x)-\gamma|}{\Vert g\Vert_{L^s(\sigma B_0)}}\leq 
\widetilde{C}C_1(k-k_0)
\quad
\text{for $k>k_0$ and all $x\in E_k\cap B_0$}.
\end{equation}
Take a constant $C_1$ in such a way that $\exp(\widetilde{C}C_1)=2^s$.

Let us split the integral that we need to estimate into two integrals
$$
\mvint_{B_0}\exp\left(\frac{C_1 b^{1/s}|u-\gamma|}{\Vert g\Vert_{L^s(\sigma B_0)}}\right)\, d\mu
=\frac{1}{\mu(B_0)}\int_{B_0\cap E_{k_0}} + \frac{1}{\mu(B_0)}\int_{B_0\setminus E_{k_0}}
= I_1+I_2.
$$
Estimate \eqref{eq26} gives
$$
I_1\leq\frac{\mu(B_0\cap E_{k_0})}{\mu(B_0)}\exp(CC_1)\leq\exp(CC_1),
$$
while estimate \eqref{eq27} and the fact that $\exp(\widetilde{C}C_1)=2^s$ yield
\begin{equation*}
\begin{split}
I_2
&\leq
\frac{1}{\mu(B_0)}\sum_{k=k_0+1}^\infty 
\exp(\widetilde{C}C_1(k-k_0))\mu\big(B_0\cap(E_k\setminus E_{k-1})\big)\\
&\leq \frac{2^{-sk_0}}{\mu(B_0)}\sum_{k=-\infty}^\infty 2^{sk}\mu(E_k\setminus E_{k-1})
\leq
C\frac{2^{-sk_0}}{\mu(B_0)}\int_{\sigma B_0} g^s\, d\mu 
\leq
C\frac{bR_0^s}{\mu(B_0)}\leq C,
\end{split}
\end{equation*}
where the last two estimates follow from \eqref{eq16} and the volume condition
$V(\sigma B_0,s,b)$, respectively. The proof in the case $p=s$ is now complete.

\noindent
{\em Proof of (c).}
Let $\gamma=u(y)$ be as in \eqref{eq17} and
$a_k=\sup_{E_k\cap B_0}|u-\gamma|$. Since $2^{1-p/s}<1$, 
we can estimate the finite sum at \eqref{eq44} for $k> k_0$, by the convergent geometric series
$\sum_{j=k_0}^\infty2^{j(1-p/s)}=C2^{k_0(1-p/s)}$. As such,
\begin{equation}
\label{eq28}
a_k\leq Cb^{-1/s}\left(\, \int_{\sigma B_0} g^p\, d\mu\right)^{\!\!\!1/s} 2^{k_0(1-p/s)}
+\sup_{E_{k_0}\cap\sigma B_0}|u-\gamma|
\quad\text{for $k> k_0$.}
\end{equation}
Then \eqref{eq28}, \eqref{eq16}, and \eqref{eq17} yield
$$
a_k\leq C(bR_0^s)^{-1/p}R_0\left(\, \int_{\sigma B_0}g^p\, d\mu\right)^{\!\!\!1/p}=
C\left(\frac{\mu(\sigma B_0)}{bR_0^s}\right)^{\!\!1/p} R_0
\left(\, \mvint_{\sigma B_0} g^p\, d\mu\right)^{\!\!\!1/p}
\quad
\text{for $k>k_0$.}
$$
Since the right hand side is a constant that does not depend on $k$, we conclude that $|u-\gamma|$ is bounded on $B_0$. More precisely, $|u-\gamma|$ equals almost everywhere to a function that is bounded in $B_0$ and \eqref{eq29} follows from the estimate
$$
\Vert u-u_{B_0}\Vert_{L^\infty(B_0)}\leq 2\Vert u-\gamma\Vert_{L^\infty(B_0)}.
$$

It remains to prove H\"older continuity of $u$ along with the estimate \eqref{eq30}.
If $x,y\in B_0$ and $R_1:=2d(x,y)\leq (\sigma-1)R_0/\sigma$, then
$x,y\in B_1:=B(x,R_1)$ and $\sigma B_1\subseteq\sigma B_0$. Therefore, estimate \eqref{eq29} applied to $B_1$ in place of $B_0$ yields
\begin{equation*}
\begin{split}
|u(x)-u(y)|
&\leq 
2\Vert u-u_{B_1}\Vert_{L^\infty(B_1)}\leq
C\left(\frac{\mu(\sigma B_1)}{bR_1^s}\right)^{\!\!1/p}R_1\, \left(\, \mvint_{\sigma B_1}g^p\, d\mu\right)^{\!\!\!1/p}\\
&=
Cb^{-1/p}d(x,y)^{1-s/p}\left(\, \int_{\sigma B_1} g^p\, d\mu\right)^{\!\!\!1/p}.
\end{split}
\end{equation*}
If $2d(x,y)>(\sigma-1)R_0/\sigma$, then \eqref{eq29} gives
\begin{equation*}
\begin{split}
|u(x)-u(y)|
&\leq 
2\Vert u-u_{B_0}\Vert_{L^\infty(B_0)}\leq
C\left(\frac{\mu(\sigma B_0)}{bR_0^s}\right)^{\!\!1/p}R_0\, \left(\, \mvint_{\sigma B_0}g^p\, d\mu\right)^{\!\!\!1/p}\\
&\leq
Cb^{-1/p}d(x,y)^{1-s/p}\left(\, \int_{\sigma B_0} g^p\, d\mu\right)^{\!\!\!1/p}.
\end{split}
\end{equation*}
This completes the proof of Theorem~\ref{embedding}.
\end{proof}

\section{Auxiliary Results}
\label{S3}

In this section we will collect some lemmata of a purely technical nature that will be needed in the proofs of the main results. 
Since the results collected here are not interesting on its own, the reader may skip this section for now and return to it when needed.

An open set $\Omega\subseteq\bbbr^n$ is a metric-measure space with the Euclidean metric and Lebesgue measure.
If $x\in\Omega$ and $r\in(0,\infty)$, then we can always find a radius $r_x<r$ such that
$|B(x,r_x)\cap\Omega|=\frac{1}{2}|B(x,r)\cap\Omega)|$. However, in a general metric-measure space $(X,d,\mu)$, it is not always possible to
find a concentric ball with half of the measure of the original ball, but for $x\in X$ and $r\in[0,\infty)$ we still can define
\begin{equation*}
\vi_x(r)=\sup\bigg\{s\in [0,r]:\, \mu(B(x,s))\leq\frac{1}{2}\mu(B(x,r))\bigg\}.
\end{equation*}
Note that for $s=0$, $B(x,s)=\emptyset$ so $\varphi_x(r)\geq 0$.
The basic properties of $\vi_x(r)$ are listed in the next lemma. The reader is reminded that
$\overline{B}(x,r):=\{y\in X:\,d(x,y)\leq r\}$, $x\in X$, $r\in[0,\infty)$.
In particular $\overline{B}(x,0)=\{ x\}$.

\begin{lemma}\label{Gds.24}
Suppose that $(X,d,\mu)$ is a metric-measure space and
fix $x\in X$, $r\in[0,\infty)$.
Then the following statements hold.
\begin{enumerate}[(i)]
\item $\varphi_x(\cdot)$ is nondecreasing, i.e., $\varphi_x(s)\leq\varphi_x(t)$ 
whenever $0\leq s\leq t<\infty$.
\vskip.08in
\item One has that
\begin{equation}
\label{LLk-3}
\mu\big(B(x,\varphi_x(r))\big)\leq\frac{1}{2}\,\mu\big(B(x,r)\big)
\leq\mu\big(\overline{B}(x,\varphi_x(r))\big).
\end{equation} 

\vskip.08in

\item $\varphi_x(r)\in[0,r]$, where $\varphi_x(r)=r$ if and only $r=0$.
\vskip.08in
\item If $\mu\big(\{x\}\big)=0$ and $r>0$, then $\varphi^j_x(r)>0$ for every $j\in\mathbb{N}_0$,
where
$$
\varphi^0_x(r):=r\,\,\mbox{ and }\,\,\,
\varphi^j_x(r):=\varphi_x\big(\varphi^{j-1}_x(r)\big),\,\, j\in\mathbb{N}.
$$
Moreover, the sequence $\{\varphi^j_x(r)\}_{j\in\mathbb{N}_0}$ is strictly decreasing, i.e.,
\begin{equation}
\label{uu-12}
r>\varphi_x(r)>\varphi_x^2(r)>\cdots>\varphi^j_x(r)>
\varphi^{j+1}_x(r)>\cdots>0,
\end{equation}
and $\mu\big(B(x,\varphi_x^j(r))\big)\leq 2^{-j}\,\mu\big(B(x,r)\big)$. Consequently,
$\lim\limits_{j\to\infty}\varphi^j_x(r)=0$. 
\end{enumerate}
\end{lemma}

\begin{proof}
Given that {\it (i)} follows immediately from the definition of 
$\varphi_x$,
we begin by establishing the first inequality in \eqref{LLk-3}.
Take $r_j\uparrow\varphi_x(r)$. Then $\mu(B(x,r_j))\leq \frac{1}{2}\mu(B(x,r))$ for each $j$ and hence,
$\mu(B(x,\varphi_x(r)))=\lim_{j\to\infty}\mu(B(x,r_j))\leq\frac{1}{2}\mu(B(x,r))$. To prove the second inequality in \eqref{LLk-3}, 
observe 
that $\frac{1}{2}\mu(B(x,r))<\mu(B(x,s))$ for all $s>\varphi_x(r)$. 
Indeed, if $\varphi_x(r)<r$, this follows from the definition of $\varphi_x(r)$; if $\varphi_x(r)=r$, the first inequality in \eqref{LLk-3} implies that $r=0$ and the estimate is obvious.
Now passing to the limit in $\frac{1}{2}\mu(B(x,r))\leq\mu(B(x,s))$
as $s$ is decreasing to $\varphi_x(r)$, yields the second inequality in \eqref{LLk-3}.
This completes the proof of {\it (ii)}. The claim {\it (iii)} easily follows from the first inequality in \eqref{LLk-3}.

As concerns {\it (iv)}, it is clear that $r>\varphi_x(r)>0$ given
{\it (iii)} and the second inequality in \eqref{LLk-3}. Then \eqref{uu-12} 
can now be justified using an inductive argument. Finally, 
repeatedly calling upon \eqref{LLk-3} we have
$\mu\big(B(x,\varphi_x^j(r))\big)\leq 2^{-j}\,\mu\big(B(x,r)\big)$,
from which it follows that $\lim_{j\to\infty}\varphi^j_x(r)=0$, 
This finishes the proof of the lemma.
\end{proof}

Recall that a metric space $(X,d)$ is said to be {\tt uniformly} {\tt perfect} if there exists a constant $\lambda\in(0,1)$ with the property that
for each $x\in X$ and each $r\in(0,\infty)$ one has
\begin{equation}
\label{U-perf}
B(x,r)\setminus B(x,\lambda r)\neq\emptyset\quad
\mbox{ whenever }\quad X\setminus B(x,r)\neq\emptyset.
\end{equation}
Note that every connected space is uniformly perfect; 
however, the class of uniformly perfect spaces contains very disconnected sets such as the Cantor set. Moreover, observe that if \eqref{U-perf} holds for some $\lambda\in(0,1)$ then it holds for every $\lambda'\in(0,\lambda]$. Therefore, we may always assume that $0<\lambda<1/5$.

Since by our assumptions metric spaces have at least two points, it easily follows that uniformly perfect spaces have no isolated points.
\begin{lemma}
\label{tx}
Let $(X,d,\mu)$ be a uniformly perfect metric-measure space and let $0<\lambda<1/5$ be as in \eqref{U-perf}. If $x\in X$, $r\in \big(0,{\rm \diam}(X)\big]$ is finite, and $r>3\varphi_x(r)/\lambda^2$, then there is a ball $B(\widetilde{x},\widetilde{r})\subseteq B(x,r)$ such that
$\lambda r<\widetilde{r}\leq\min\{ r,3\varphi_{\widetilde{x}}(\widetilde{r})/\lambda^2\}$.
\end{lemma}
\begin{proof}
Note that $X\setminus B\big(x,\varphi_x(r)/\lambda+2\lambda r\big)\neq\emptyset$.
Indeed, given that $\lambda<1/5$, the radius of the ball is less than $7r/15$ and hence, its diameter is less than 
$14r/15<\diam (X)$.
Since $(X,d)$ is uniformly perfect, we may choose a point
$$
\widetilde{x}\in B\big(x,\varphi_x(r)/\lambda+2\lambda r\big)\setminus B\big(x,\varphi_x(r)+2\lambda^2 r\big).
$$
With
$\widetilde{r}:=2\varphi_x(r)/\lambda+2\lambda r>\lambda r$,
we claim that
\begin{equation}
\label{dx.1}
\overline{B}(x,\varphi_x(r))\subseteq B(\widetilde{x},\widetilde{r})\subseteq B(x,r)
\quad 
\text{and}
\quad
B(\widetilde{x},2^{-1}\lambda \widetilde{r})\subseteq B(x,r)\setminus
\overline{B}(x,\varphi_x(r)).
\end{equation}
For the inclusion $\overline{B}(x,\varphi_x(r))\subseteq B(\widetilde{x},\widetilde{r})$, observe that if $z\in \overline{B}(x,\varphi_x(r))$, then
$$
d(z,\widetilde{x})\leq d(z,x)+d(x,\widetilde{x})
<\varphi_x(r)+[\varphi_x(r)/\lambda+2\lambda r]\leq \widetilde{r},
$$
given that $1/\lambda>1$. To prove  
$B(\widetilde{x},\widetilde{r})\subseteq B(x,r)$, observe that for $z\in B(\widetilde{x},\widetilde{r})$, we have
\begin{equation*}
\begin{split}
d(z,x)\leq d(z,\widetilde{x})+d(\widetilde{x},x)
&< [2\varphi_x(r)/\lambda+2\lambda r]+[\varphi_x(r)/\lambda+2\lambda r]\\
&=
3\varphi_x(r)/\lambda+4\lambda r<5\lambda r<r.
\end{split}
\end{equation*}
To finish the proof of \eqref{dx.1} we need to show that 
$B(\widetilde{x},2^{-1}\lambda  \widetilde{r})\subseteq B(x,r)\setminus\overline{B}(x,\varphi_x(r))$. 
From what we have already proved $B(\widetilde{x},2^{-1}\lambda \widetilde{r})\subseteq B(\widetilde{x},\widetilde{r})\subseteq B(x,r)$.
To show that
$B(\widetilde{x},2^{-1}\lambda  \widetilde{r})\subseteq X\setminus\overline{B}(x,\varphi_x(r))$, fix $z\in B(\widetilde{x},2^{-1}\lambda \widetilde{r})$. Then,
$$
\varphi_x(r)+2\lambda^2 r\leq d(x,\widetilde{x})
\leq d(x,z)+d(z,\widetilde{x})
< d(x,z)+2^{-1}\lambda \widetilde{r},
$$
$$
d(z,x)> \varphi_x(r)+2\lambda^2 r-2^{-1}\lambda \widetilde{r}
=\lambda^2 r>3\varphi_x(r)\geq \varphi_x(r),
$$
which in turn implies the desired inclusion. This finishes
the proof of \eqref{dx.1}. 

It follows from \eqref{dx.1} that $\overline{B}(x,\varphi_x(r))$ and $B(\widetilde{x},2^{-1}\lambda \widetilde{r})$ are disjoint subsets of $B(\widetilde{x},\widetilde{r})$ and
\begin{equation}
\label{rn-2}
\begin{split}
\mu\big(B(\widetilde{x},2^{-1}\lambda \widetilde{r})\big)&=\frac{1}{2}\big[\mu\big(B(\widetilde{x},2^{-1}\lambda \widetilde{r})\big)+\mu\big(B(\widetilde{x},2^{-1}\lambda \widetilde{r})\big)\big]\\
&\leq\frac{1}{2}\big[\mu\big(B(x,r)\setminus
\overline{B}(x,\varphi_x(r))\big)+\mu\big(B(\widetilde{x},2^{-1}\lambda \widetilde{r})\big)\big]\\
&\leq\frac{1}{2}\big[\mu\big(\overline{B}(x,\varphi_x(r))\big)+\mu\big(B(\widetilde{x},2^{-1}\lambda \widetilde{r})\big)\big]
\leq\frac{1}{2}\,\mu\big(B(\widetilde{x},\widetilde{r})\big),
\end{split}
\end{equation}
where, in obtaining the second inequality in \eqref{rn-2}, we have used
second inequality in \eqref{LLk-3}.
Inequality \eqref{rn-2} implies that $\varphi_{\widetilde{x}}(\widetilde{r})\geq 2^{-1}\lambda \widetilde{r}$. 
Hence, $\widetilde{r}\leq 2\varphi_{\widetilde{x}}(\widetilde{r})/\lambda\leq 3\varphi_{\widetilde{x}}(\widetilde{r})/\lambda^2$, since $1/\lambda>1$.  On the other hand, it is straightforward to verify that 
$$
\widetilde{r}=\frac{2\varphi_x(r)}{\lambda}+2\lambda r<\frac{2\lambda r}{3} + 2\lambda r<r.
$$
The proof is complete.
\end{proof}
\begin{lemma}
\label{en2-4}
Let $(X,d,\mu)$ be a uniformly perfect metric-measure space, fix
$s\in (0,\infty)$, and let $0<\lambda<1/5$ be as in \eqref{U-perf}.
\begin{enumerate}
\item[(i)]
Assume that there is a finite constant $C>0$ such that
$\mu\big(B(x,r)\big)\geq Cr^s$ whenever $x\in X$ and
finite $r\in\big(0,{\rm diam}(X)\big]$ satisfy 
$r\leq3\varphi_x(r)/\lambda^2$. Then
$\mu\big(B(x,r)\big)\geq \widetilde{C}r^s$ 
for every $x\in X$ and every finite $r\in\big(0,{\rm diam}(X)\big]$, 
where
$\widetilde{C}=C\lambda^{s}$.
\item[(ii)]
Assume that there is a finite constant $C>0$ such that
\begin{equation}
\label{eq48}
\frac{\mu(B(x,r))}{\mu(B(y,R))}\geq C\left(\frac{r}{R}\right)^s,
\end{equation}
whenever $x,y\in X$, $B(x,r)\subseteq B(y,R)$, $0<r\leq R<\infty$, and
$r\leq 3\varphi_x(r)/\lambda^2$. 

If $\widetilde{C}=C\lambda^s$, then
\begin{equation}
\label{eq49}
\frac{\mu(B(x,r))}{\mu(B(y,R))}\geq \widetilde{C}\left(\frac{r}{R}\right)^s,
\end{equation}
whenever $x,y\in X$, $B(x,r)\subseteq B(y,R)$, and $0<r\leq R<\infty$.
\end{enumerate}
\end{lemma}
\begin{proof}
In order to prove {\it (i)}, fix a point $x\in X$, a finite radius
$r\in\big(0,{\rm diam}(X)\big]$.
If  $r\leq 3\varphi_x(r)/\lambda^2$,  
then $\mu\big(B(x,r)\big)\geq Cr^s>\widetilde{C} r^s$ by assumption. Thus, in what follows we will also assume that $r>3\varphi_x(r)/\lambda^2$.
Let $B(\widetilde{x},\widetilde{r})\subseteq B(x,r)$ be a ball as in Lemma~\ref{tx}. Since
$\widetilde{r}\leq 3\varphi_{\widetilde{x}}(\widetilde{r})/\lambda^2$ and $\lambda r<\widetilde{r}\leq r\leq\diam (X)$, it follows that
$$
\mu(B(x,r))\geq \mu(B(\widetilde{x},\widetilde{r}))\geq  C\widetilde{r}^s\geq C\lambda^s r^s.
$$
It remains to prove {\it (ii)}. 
First observe that \eqref{eq48} implies that $C\leq 1$. Let $z\in X$ and $0<r<\diam(X)$. Since there are no isolated points in $X$, there are infinitely many points in $B(z,r/2)$ and so we can find $x\in B(z,r/2)$ with measure as small as we wish. In particular, we may find $x\in B(z,r/2)$ such that
\begin{equation}
\label{eq50}
\mu(\{ x\})<\frac{1}{2}\mu(B(z,r/2))\leq\frac{1}{2}\mu(B(x,r)),
\end{equation}
where the last inequality is a consequence of the inclusion 
$B(z,r/2)\subseteq B(x,r)$. It follows that $\varphi_x(r)>0$ as otherwise we would have $\overline{B}(x,\varphi_x(r))=\{ x\}$ and \eqref{eq50} would contradict the second inequality in \eqref{LLk-3}. Given that $\varphi_x(r)>0$, we can take $y=x$ and $0<r=R\leq 3\varphi_x(r)/\lambda^2$. Then \eqref{eq48} readily implies that $C\leq 1$.

Now we can finish the proof of {\it (ii)}.
Let $x,y\in X$, $B(x,r)\subseteq B(y,R)$, and $0<r\leq R<\infty$. 
We want to prove \eqref{eq49}. If $r\leq 3\varphi_x(r)/\lambda^2$, then the claim follows by assumptions. Thus we may assume that $r>3\varphi_x(r)/\lambda^2$.
If $r>\diam(X)$, then $B(x,r)=B(x,R)=X$ and \eqref{eq49} is trivially true, given that $1\geq C>\widetilde{C}$. Thus we may assume that $r\leq\diam(X)$. This allows us to find a ball $B(\widetilde{x},\widetilde{r})\subseteq B(x,r)$ as in Lemma~\ref{tx}.
Since $B(\widetilde{x},\widetilde{r})\subseteq B(y,R)$, $\widetilde{r}\leq r$, it follows that $0<\widetilde{r}\leq R$ and $\widetilde{r}\leq 3\varphi_{\widetilde{x}}(\widetilde{r})/\lambda^2$. Therefore,
\eqref{eq48} implies
$$
\frac{\mu(B(x,r))}{\mu(B(y,R))}\geq
\frac{\mu(B(\widetilde{x},\widetilde{r}))}{\mu(B(y,R))}\geq
C\left(\frac{\widetilde{r}}{R}\right)^s> C\lambda^s\left(\frac{r}{R}\right)^s, 
$$
where in the proof of the last inequality we used the estimate $\widetilde{r}>\lambda r$.
This finishes the proof of the lemma.
\end{proof}

In the sequel, we will also need the following well-known result. 

\begin{lemma}
\label{Lipbump}
Given $x\in X$ and $0\leq r<R<\infty$, there exists a $(R-r)^{-1}$-Lipschitz function $\Phi_{r,R}:X\to[0,1]$ such that $\Phi_{r,R}\equiv1$ on $\overline{B}(x,r)$ and\, $\Phi_{r,R}\equiv0$ on $X\setminus B(x,R)$. Consequently, one has $(R-r)^{-1}\chi_{B(x,R)}\in D\big(\Phi_{r,R}\big)$.
\end{lemma}

\begin{proof}
Fix a point $x\in X$, numbers $0\leq r<R<\infty$, and define $\Phi_{r,R}:X\to[0,1]$ by setting for each $y\in X$,
\begin{eqnarray*}
\Phi_{r,R}(y):=
\left\{
\begin{array}{ll}
\,\qquad 1\quad &\mbox{if $y\in \overline{B}(x,r)$,}
\\[6pt]
\displaystyle\frac{R-d(x,y)}{R-r}
&\mbox{if $y\in B(x,R)\setminus B(x,r)$,}
\\[15pt]
\,\qquad 0 &\mbox{if $y\in X\setminus B(x,R)$.}
\end{array}
\right.
\end{eqnarray*}
Then the claims follow from straightforward computations.
\end{proof}

\begin{construction}
\label{Cons1}
Let $(X,d,\mu)$ be {\it any} metric-measure space and $\sigma\in[1,\infty)$. Fix a ball $B:=B(x,r)$ with $x\in X$ 
and $r\in\big(0,{\rm diam}(X)\big]$, finite.
We define a collection of functions $\{u_j\}_{j\in\mathbb{N}}$
as follows: for each $j\in\mathbb{N}$, let 
$r_j:=(2^{-j-1}+2^{-1})r$ and set $B^j:=B(x,r_j)$. Then
$$
\frac{1}{2}r<r_{j+1}<r_j\leq\frac{3}{4}r,
\quad\forall\,j\in\mathbb{N}.
$$
Then for each $j\in\mathbb{N}$, define $u_j:X\to[0,1]$ by setting 
$u_j:=\Phi_{r_{j+1},r_j}$,
where the function $\Phi_{r_{j+1},r_j}$ is as in Lemma~\ref{Lipbump}.
Noting that $\displaystyle (r_j-r_{j+1})^{-1}=2^{j+2}r^{-1}$, we have that $u_j$ is $2^{j+2}r^{-1}$-Lipschitz on $X$ supported in $B^j$, and that $g_j:=2^{j+2}r^{-1}\chi_{B^j}\in D(u_j)$.
In particular, we have that $u_j\in M^{1,p}(\sigma B,d,\mu)$.
\hfill$\blacksquare$
\end{construction}

\begin{construction}
\label{Cons2}
Let $(X,d,\mu)$ be a uniformly perfect metric-measure space, suppose $\sigma\in[1,\infty)$, and let $0<\lambda<1/5$ be as in \eqref{U-perf}.
Fix a ball $B:=B(x,r)$ with $x\in X$  and $r\in\big(0,{\rm diam}(X)\big]$, finite. Assume that
$r\leq 3\varphi_x(r)/\lambda^2$. 
In particular, we have $0<\varphi_x(r)<r$ (see part {\it (iii)} of Lemma~\ref{Gds.24} for the second inequality).
Define a collection of 
Lipschitz functions $\{\widetilde{u}_j\}_{j\in\mathbb{N}}$
by first considering radii $\widetilde{r}_j:=(2^{-j-1}+2^{-1})\varphi_x(r)$, $j\in\mathbb{N}$ which
satisfy
\begin{equation}
\label{JG-924}
\frac{1}{2}\varphi_x(r)<\widetilde{r}_{j+1}<\widetilde{r}_j\leq\frac{3}{4}\varphi_x(r).
\end{equation}
For each $j\in\mathbb{N}$, let $\widetilde{u}_j:X\to[0,1]$ be defined by
$\widetilde{u}_j:=\Phi_{\widetilde{r}_{j+1},\widetilde{r}_j}$,
where the function $\Phi_{\widetilde{r}_{j+1},\widetilde{r}_j}$ is as in Lemma~\ref{Lipbump}. Then each  $\widetilde{u}_j$ is $2^{j+2}\varphi_x(r)^{-1}$-Lipschitz on $X$, and
$\widetilde{g}_j:=2^{j+2}\varphi_x(r)^{-1}\chi_{\widetilde{B}^j}\in D(\widetilde{u}_j)$, where
$\widetilde{B}^j:=B(x,\widetilde{r}_j)$. 
In particular, $\widetilde{u}_j\in M^{1,p}(\sigma B,d,\mu)$.

Observe that $\widetilde{u}_j\equiv1$ on $\widetilde{B}^{j+1}$ and $\widetilde{u}_j\equiv0$ on $B\setminus \widetilde{B}^j$. It follows that
for each $\gamma\in\mathbb{R}$ we have $|\widetilde{u}_j-\gamma\,|\geq\frac{1}{2}$ on at least one of the sets 
$\widetilde{B}^{j+1}$ and $B\setminus \widetilde{B}^j$. Observe that by
combining \eqref{LLk-3} in Lemma~\ref{Gds.24} and \eqref{JG-924}, we have
$$
\mu\big(\widetilde{B}^{j+1}\big)\leq\mu\big(B(x,\varphi_x(r))\big)\leq\frac{1}{2}\mu(B),
$$
and
$$
\mu\big(B\setminus \widetilde{B}^{j}\big)=\mu(B)-\mu\big(\widetilde{B}^{j}\big)
\geq\mu(B)-\mu\big(B(x,\varphi_x(r))\big)\geq\frac{1}{2}\mu(B).
$$
Therefore,
$$
\min\big\{\mu\big(\widetilde{B}^{j+1}\big),\,\mu\big(B\setminus \widetilde{B}^j\big)\big\} =\mu\big(\widetilde{B}^{j+1}\big)
$$
and hence, 
\begin{equation}
\label{eq51}
\text{$|\widetilde{u}_j-\gamma\,|\geq\frac{1}{2}$ on a subset of $B$ having $\mu$-measure at least  $\mu\big(\widetilde{B}^{j+1}\big)$.}
\end{equation}
\hfill$\blacksquare$
\end{construction}

The next result is an abstract iteration scheme that will be applied many times in the proofs that the embedding theorem imply the measure condition. It is an abstract version of an argument used in \cite{korobenkomr}.

\begin{lemma}
\label{iteration}
Suppose $0<a<b<\infty$, $0<p<q<\infty$  and $\rho,\tau\in (0,\infty)$. If a sequence $(a_j)_{j\in\bbbn}$ satisfies
\begin{equation}
\label{eq46}
a\leq a_j\leq b
\quad
\text{and}
\quad
a_{j+1}^{1/q}\leq\rho \tau^j a_j^{1/p}
\qquad
\forall\ j\in\bbbn,
\end{equation}
then
$$
a_1^{1-p/q}\,\rho^p\,\tau^{pq/(q-p)}\geq 1.
$$
\end{lemma}
\begin{proof}
Let $\alpha:=p/q\in (0,1)$. Rise both sides of the second inequality in \eqref{eq46} to the power $p\alpha^{j-1}$. Then
$$
a_{j+1}^{\alpha^j}\leq\rho^{p\alpha^{j-1}}\tau^{pj\alpha^{j-1}}a_j^{\alpha^{j-1}}.
$$
With $P_j=a_j^{\alpha^{j-1}}$, the above inequality reads as
$$
P_{j+1}\leq\rho^{p\alpha^{j-1}}\tau^{pj\alpha^{j-1}} P_j,
$$
from which a simple induction argument and an observation that $P_1=a_1$ give
\begin{equation}
\label{eq47}
P_{j+1}\leq a_1\prod_{k=1}^j\left[\rho^{p\alpha^{k-1}}\tau^{pk\alpha^{k-1}}\right].
\end{equation}
Since $a_j\in [a,b]$ and $\alpha^{j-1}\to 0$ as $j\to\infty$, it follows that $\lim_{j\to\infty} P_j=1$. Therefore, passing to the limit in \eqref{eq47} as $j\to\infty$ gives
\begin{equation*}
\begin{split}
1&\leq
a_1\prod_{k=1}^\infty \left[\rho^{p\alpha^{k-1}}\tau^{pk\alpha^{k-1}}\right] 
=
a_1\rho^{p\sum_{k=1}^\infty \alpha^{k-1}}\tau^{p\sum_{k=1}^\infty k\alpha^{k-1}}\\
&=
a_1\rho^{p/(1-\alpha)}\tau^{p/(1-\alpha)^2}=
a_1\rho^{pq/(q-p)}\tau^{pq^2/(q-p)^2}
\end{split}
\end{equation*}
and the lemma easily follows.
\end{proof}

\section{The Case $p<s$}
\label{S4}

\begin{theorem}
\label{PlessS}
Fix $\sigma\in(1,\infty)$, $s\in(0,\infty)$, $p\in(0,s)$, and let $p^*=sp/(s-p)$.
Then the following statements are equivalent.
\begin{enumerate}[(a)]
\item There exists a constant $\kappa\in(0,\infty)$ such
that
\begin{equation}
\label{measboundthm-X}
\mu\big(B(x,r)\big)\geq \kappa\,r^s
\quad
\text{for every $x\in X$ and every finite $r\in\big(0,{\rm diam}(X)\big].$}
\end{equation}
\item There exists a constant $C_S\in(0,\infty)$ such that
for every ball $B_0:=B(x_0,R_0)$ with $x_0\in X$ and finite $R_0\in\big(0,{\rm diam}(X)\big]$,  one has
\begin{equation}
\label{eq-JK-X}
\left(\, \mvint_{B_0} |u|^{p^*}\, d\mu\right)^{\!\!\!1/p^*}\!\!\!\leq
C_S\bigg(\frac{\mu(\sigma B_0)}{R_0^s}\bigg)^{\!\!1/p}\!
\left[R_0\left(\,\mvint_{\sigma B_0} g^p\, d\mu\right)^{\!\!\!1/p}
\!\!+\left(\,\mvint_{\sigma B_0}|u|^p\, d\mu\right)^{\!\!\!1/p}\right],
\end{equation}
whenever $u\in M^{1,p}(\sigma B_0,d,\mu)$ and $g\in D(u)$.
\end{enumerate}

If, in addition, $(X,d)$ is assumed to be uniformly perfect $($cf.~\eqref{U-perf}$)$ then {\it (a)} $($hence, also {\it (b)}$)$ is further 
equivalent to:
\begin{enumerate}
\item[(c)] There exists a constant $C_P\in(0,\infty)$ such that
for every ball $B_0:=B(x_0,R_0)$ with $x_0\in X$ and  finite $R_0\in\big(0,{\rm diam}(X)\big]$, one has
\begin{equation}\label{HHs-175}
\inf_{\gamma\in\bbbr}\left(\, \mvint_{B_0} |u-\gamma|^{p^*}\, d\mu\right)^{\!\!\!1/p^*}\leq
C_P\bigg(\frac{\mu(\sigma B_0)}{R_0^s}\bigg)^{\!\!1/p}R_0\left(\,\mvint_{\sigma B_0} g^p\, d\mu\right)^{\!\!\!1/p},
\end{equation}
whenever $u\in M^{1,p}(\sigma B_0,d,\mu)$ and $g\in D(u)$.
\end{enumerate}
\end{theorem}
\begin{remark}
The implications {\it (a)} $\Longleftrightarrow$ {\it (b)} and {\it (a)} $\Rightarrow$  {\it (c)} will be proven without assuming that $(X,d)$ is uniformly perfect. Uniform perfectness will only be needed in the proof of the implication {\it (c)} $\Rightarrow$ {\it (a)}. 
\hfill$\blacksquare$
\end{remark}
\begin{remark}
As the proofs of the implications {\it (b)} $\Rightarrow$
{\it (a)} and {\it (c)} $\Rightarrow$
{\it (a)} will reveal, one can take the constants $\kappa$
in \eqref{measboundthm-X} to be $2^{-s}(8C_S)^{-p}$ and 
$2^{-s}(24C_P\lambda^{-2})^{-p}\lambda^s$, respectively. Here, $\lambda<1/5$ is the constant from the definition of the uniformly perfect space (see \eqref{U-perf}). Moreover, the implications {\it (b)} $\Rightarrow$
{\it (a)} and {\it (c)} $\Rightarrow$
{\it (a)} are also valid when $\sigma=1$.
\hfill$\blacksquare$
\end{remark}

\begin{proof}
We will first show that {\it (a)} implies both {\it (b)}
and {\it (c)} (without the additional assumption that $X$
is uniformly perfect). Fix a ball $B_0$ having finite radius $R_0\in\big(0,{\rm diam}(X)\big]$. 
If $B(x,r)\subseteq\sigma B_0$ and $r\in (0,\sigma R_0]$, then $\sigma^{-1}r\leq{\rm diam}(X)$, and 
\eqref{measboundthm-X} yields
\begin{equation}
\label{eq52}
\mu(B(x,r))\geq\mu(B(x,\sigma^{-1}r))\geq\kappa(\sigma^{-1}r)^s=c'r^s,
\end{equation}
where $c'=\kappa\sigma^{-s}$. Thus, $\mu$ satisfies the $V(\sigma B_0,s,c')$ condition. Since
$$
(c')^{-1/p}\left(\frac{\mu(\sigma B_0)}{R_0^s}\right)^{1/p}\geq
\left(\frac{c'(\sigma R_0)^s}{c'R_0^s}\right)^{1/p}=\sigma^{s/p}>1,
$$
inequality \eqref{eq18} with $b$ replaced by $c'$ yields
\begin{equation*}
\left(\, \mvint_{B_0} |u|^{p^*}\, d\mu\right)^{\!\!\!1/p^*}\leq 
C(c')^{-1/p}\bigg(\frac{\mu(\sigma B_0)}{R_0^s}\bigg)^{\!\!1/p}\!
\left[R_0\left(\,\mvint_{\sigma B_0} g^p\, d\mu\right)^{\!\!\!1/p}
\!\!+\left(\,\mvint_{\sigma B_0}|u|^p\, d\mu\right)^{\!\!\!1/p}\right].
\end{equation*}
Hence, {\it (b)} is valid. 
Given that \eqref{HHs-175} is an immediate consequence of \eqref{eq19}, this finishes the proof of the fact that {\it (a)} implies both {\it (b)} and {\it (c)}.

We now focus on proving that {\it (b)} implies {\it (a)} (still without assuming that the space is uniformly perfect).
To this end, fix a ball $B:=B(x,r)$ with $x\in X$ 
and $r\in\big(0,{\rm diam}(X)\big]$, finite. Specializing
\eqref{eq-JK-X} to the case when $B_0:=B$ (and simplifying the expression) implies that
\begin{equation}
\label{eq-JK2-X}
\left(\, \mvint_B |u|^{p^*}\, d\mu\right)^{\!\!\!1/p^*}\leq
C_Sr\left(\frac{1}{r^s}\,\int_{\sigma B} g^p\, d\mu\right)^{\!\!\!1/p} 
+ C_S\left(\frac{1}{r^s}\,\int_{\sigma B}|u|^p\, d\mu\right)^{\!\!\!1/p},
\end{equation}
holds whenever $u\in M^{1,p}(\sigma B,d,\mu)$ and $g\in D(u)$. 

Let $r_j$, $u_j$, $g_j$, and $B^j$, $j\in\bbbn$ be as in Construction~\ref{Cons1}.
Since $u_j\in M^{1,p}(\sigma B,d,\mu)$, 
the functions $u_j$ and $g_j$ satisfy \eqref{eq-JK2-X}.
Observe that for each $j\in\mathbb{N}$, we have
(keeping in mind that $\sigma>1$)
\begin{equation}\label{HD-1-X}
C_Sr\left(\frac{1}{r^s}\,\int_{\sigma B} g_j^p\, d\mu\right)^{\!\!\!1/p}
= \frac{C_S 2^{j+2}}{r^{s/p}}\,\mu(B^j)^{1/p}
\end{equation}
and
\begin{equation}\label{HD-2-X}
C_S\left(\frac{1}{r^{s}}\,\int_{\sigma B}|u_j|^p\, d\mu\right)^{\!\!\!1/p}
\leq \frac{C_S}{r^{s/p}}\,\mu(B^j)^{1/p}.
\end{equation}
Moreover, since $u_j\equiv1$ on $B^{j+1}$ we may estimate
\begin{equation}\label{HD-3-X}
\left(\, \mvint_{B} |u_j|^{p^*}\, d\mu\right)^{\!\!\!1/p^*}\geq\left(\frac{\mu(B^{j+1})}{\mu( B)}\right)^{\!\!1/p^*}.
\end{equation}
In concert, \eqref{HD-1-X}-\eqref{HD-3-X} and \eqref{eq-JK2-X}, give
\begin{equation*}
\left(\frac{\mu(B^{j+1})}{\mu(B)}\right)^{\!\!1/p^*}
\leq 
C_S\bigg(\frac{2^{j+2}}{r^{s/p}}+\frac{1}{r^{s/p}}\bigg)\mu(B^j)^{1/p}
\leq 
\frac{C_S2^{j+3}}{r^{s/p}}\,\mu(B^j)^{1/p},
\quad\forall\,j\in\mathbb{N}.
\end{equation*}
Therefore,
\begin{equation*}
\mu(B^{j+1})^{1/p^*}
\leq
\left(\frac{8C_S\mu(B)^{1/p^*}}{r^{s/p}}\right) 2^j\mu(B^j)^{1/p}.
\quad\forall\,j\in\mathbb{N}.
\end{equation*}
Since $0<\mu(\frac{1}{2}B)\leq\mu(B^j)\leq\mu(B)<\infty$, $j\in\mathbb{N}$, Lemma~\ref{iteration} applied to
$$
a_j:=\mu(B^j),\quad 
p:=p,\quad 
q:=p^*,\quad
\rho:=\frac{8C_S\mu(B)^{1/p^*}}{r^{s/p}}
\quad
\text{and}
\quad\tau:=2,
$$
yields
$$
1\leq \mu(B)^{1-p/p^*}\left(\frac{8C_S\mu(B)^{1/p^*}}{r^{s/p}}\right)^p 2^{pp^*/(p^*-p)}\leq
\mu(B)(8C_S)^p r^{-s} 2^s
$$
and hence $\mu(B)\geq 2^{-s}(8C_S)^{-p} r^s$. Thus,
\eqref{measboundthm-X} holds with 
$\kappa:=2^{-s}(8C_S)^{-p}$. 
Given that $\kappa\in(0,\infty)$ is independent of
the ball $B$, this finishes the proof of the
implication {\it (b)} $\Rightarrow$ {\it (a)}.

There remains to prove that {\it (c)} implies {\it (a)}
under the additional assumption that $(X,d)$ is uniformly perfect. To this end,  fix $x\in X$ and a finite radius $r\in\big(0,{\rm diam}(X)\big]$. Let $B:=B(x,r)$. 
Also, let $\lambda\in(0,1)$ be as in \eqref{U-perf} and recall that there is no loss in generality in assuming 
that $\lambda<1/5$. As such, if we appeal to part {\it (i)} in Lemma~\ref{en2-4},
then it suffices to only consider the case when $r\leq 3\varphi_x(r)/\lambda^2$. 
It is crucial for the argument below that $r\leq 3\varphi_x(r)/\lambda^2$, because it implies that $\widetilde{r}_j>\lambda^2 r/6$ so the balls $\widetilde{B}^j$ defined below have radii comparable to that of $B$. This estimate is used in \eqref{Gkw-924} and in fact, it is the only place where we use it; that allows to mimic the argument used in the proof of the implication {\it (b)} $\Rightarrow$ {\it (a)}.

Define $\widetilde{r}_j$, $\widetilde{u}_j$, $\widetilde{g}_j$, and $\widetilde{B}^j$, $j\in\bbbn$ as in Construction~\ref{Cons2}.
Since $\widetilde{u}_j\in M^{1,p}(\sigma B,d,\mu)$, the functions $\widetilde{u}_j$ and $\widetilde{g}_j$ satisfy \eqref{HHs-175} (used here with $B_0:=B$), which after a simplification gives
\begin{equation}\label{HHs-1752}
\inf_{\gamma\in\bbbr}\left(\, \mvint_{B} |\widetilde{u}_j-\gamma|^{p^*}\, d\mu\right)^{\!\!\!1/p^*}\leq
C_Pr\left(\frac{1}{r^{s}}\,\int_{\sigma B} (\widetilde{g}_j)^p\, d\mu\right)^{\!\!\!1/p},\qquad\forall\,j\in\mathbb{N}.
\end{equation}
Observe that for each $j\in\mathbb{N}$, we have
(keeping in mind $\sigma> 1$ and $r\leq 3\varphi_x(r)/\lambda^2$)
\begin{equation}
\label{Gkw-924}
C_Pr\left(\frac{1}{r^s}\,\int_{\sigma B}(\widetilde{g}_j)^p\,d\mu\right)^{\!\!\!1/p} 
=\frac{2^{j+2}C_Pr}{\varphi_x(r)}\,\bigg(\frac{\mu\big(\widetilde{B}^j\big)}{r^{s}}\bigg)^{\!\!1/p}
\leq \frac{3\cdot 2^{j+2}C_P}{\lambda^2 \,r^{s/p}}\,\mu\big(\widetilde{B}^j\big)^{1/p}.
\end{equation}
Combining \eqref{HHs-1752} and \eqref{Gkw-924} with \eqref{eq51} gives 
\begin{equation*}
\frac{1}{2}\bigg(\frac{\mu\big(\widetilde{B}^{j+1}\big)}{\mu(B)}\bigg)^{\!\!1/p^*}
\leq 2^{j}\,\frac{12C_P}{\lambda^2\,r^{s/p}}\,\mu\big(\widetilde{B}^j\big)^{1/p}.
\end{equation*}
Hence,
\begin{equation*}
\mu\big(\widetilde{B}^{j+1}\big)^{1/p^*}
\leq 
\left(\frac{24C_P\,\mu(B)^{1/p^*}}{\lambda^2 r^{s/p}}\right)\,2^{j}\,\mu\big(\widetilde{B}^j\big)^{1/p},\qquad\forall\,j\in\mathbb{N}.
\end{equation*}
Since
$$
0<\mu\big(\mbox{$\frac{1}{2}$}B(x,\varphi_x(r))\big)\leq\mu\big(\widetilde{B}^j\big)\leq \mu(B)<\infty,
$$
Lemma~\ref{iteration} applied to
$$
a_j:=\mu\big(\widetilde{B}^j\big),\ 
p:=p,\
q:=p^*,\
\rho:=\left(\frac{24C_P\,\mu(B)^{1/p^*}}{\lambda^2 r^{s/p}}\right)
\quad
\text{and}
\quad
\tau:=2,
$$
yields
$$
1\leq \mu\big(\widetilde{B}^1\big)^{1-p/p^*}
\left(\frac{24C_P\,\mu(B)^{1/p^*}}{\lambda^2 r^{s/p}}\right)^p 2^{pp^*/(p^*-p)}
\leq \mu(B)\,(24 C_P\lambda^{-2})^{p} r^{-s}\, 2^s
$$
so
$$
\mu(B)\geq 2^{-s}(24 C_P\lambda^{-2})^{-p} r^s.
$$
This inequality was proved under the assumption that $r\leq 3\varphi_x(r)/\lambda^2$ and hence part {\it (i)} of Lemma~\ref{en2-4} implies that
$$
\mu(B(x,r))\geq 2^{-s}(24 C_P\lambda^2)^{-p}\lambda^s r^s
$$
whenever $r\in (0,\diam(X)]$ is finite.
This finishes the proof of the theorem.
\end{proof}

\begin{theorem}
\label{EMT}
Fix $\sigma\in(1,\infty)$, $s\in(0,\infty)$, $p\in(0,s)$, and let $p^*=sp/(s-p)$.
Then the following statements are equivalent.
\begin{enumerate}[(a)]
\item There is a constant $\kappa\in (0,\infty)$ such that
\begin{equation}
\label{measboundthm}
\frac{\mu\big(B(x,r)\big)}{\mu\big(B(y,R)\big)}\geq \kappa\bigg(\frac{r}{R}\bigg)^{\!\!s},
\end{equation}
whenever $x,y\in X$, $B(x,r)\subseteq B(y,R)$, $0<r\leq R<\infty$.
\item There exists a constant $C_S\in(0,\infty)$ such that
for every ball $B_0:=B(x_0,R_0)$ with $x_0\in X$ and $R_0\in(0,\infty)$, one has
\begin{equation}
\label{eq-JK}
\left(\, \mvint_{B_0} |u|^{p^*}\, d\mu\right)^{\!\!\!1/p^*}\leq
C_SR_0\left(\, \mvint_{\sigma B_0} g^p\, d\mu\right)^{\!\!\!1/p} 
+ C_S\left(\, \mvint_{\sigma B_0}|u|^p\, d\mu\right)^{\!\!\!1/p},
\end{equation}
whenever $u\in M^{1,p}(\sigma B_0,d,\mu)$ and $g\in D(u)$.
\end{enumerate}
If, in addition, $(X,d)$ is assumed to be uniformly perfect $($cf.~\eqref{U-perf}$)$ then {\it (a)} $($hence, also {\it (b)}$)$ is further 
equivalent to:
\begin{enumerate}
\item[(c)] There exists a constant $C_P\in(0,\infty)$ such that
for every ball $B_0:=B(x_0,R_0)$ with $x_0\in X$ and $R_0\in(0,\infty)$, one has
\begin{equation}
\label{eq-JK-D}
\inf_{\gamma\in\bbbr}\left(\, \mvint_{B_0} |u-\gamma|^{p^*}\, d\mu\right)^{\!\!\!1/p^*}\leq
C_P R_0\left(\, \mvint_{\sigma B_0} g^p\, d\mu\right)^{\!\!\!1/p},
\end{equation}
whenever $u\in M^{1,p}(\sigma B_0,d,\mu)$ and $g\in D(u)$.
\end{enumerate}
\end{theorem}
\begin{remark}
As the proof of Theorem~\ref{EMT} will reveal, the
implication {\it (a)} $\Rightarrow$ {\it (c)} holds
in metric measure spaces which are not necessarily 
uniformly perfect. Moreover, the implications {\it (b)} $\Rightarrow$
{\it (a)} and {\it (c)} $\Rightarrow$
{\it (a)} are also valid when $\sigma=1$.
\end{remark}

\begin{proof}
We begin proving the implication {\it (a)} implies both
{\it (b)} and {\it (c)}. 
Fix any ball $B_0:=B(x_0,R_0)$ and let $B(y,R)=\sigma B_0$. Then inequality 
\eqref{measboundthm} implies that the measure $\mu$ satisfies the $V(\sigma B_0,s,b)$ condition with 
$b=\kappa\mu(\sigma B_0)(\sigma R_0)^{-s}$.
As such, for this value of $b$ the inequalities displayed in \eqref{eq-JK} and \eqref{eq-JK-D} follow immediately from 
\eqref{eq18}-\eqref{eq19} in Theorem~\ref{embedding}.
Note that these implications are valid without the additional uniformly perfect property.

We prove next that {\it (a)} follows from {\it (b)}.
Fix a ball $B_0:=B(y,R)$ so inequality \eqref{eq-JK} gives
\begin{equation}
\label{eq-JK2}
\left(\, \mvint_{B_0} |u|^{p^*}\, d\mu\right)^{\!\!\!1/p^*}\leq
C_S R\left(\, \mvint_{\sigma B_0} g^p\, d\mu\right)^{\!\!\!1/p} 
+ C_S\left(\, \mvint_{\sigma B_0}|u|^p\, d\mu\right)^{\!\!\!1/p},
\end{equation}
whenever $u\in M^{1,p}(\sigma B_0,d,\mu)$ and $g\in D(u)$.

Moving on, suppose 
$B:=B(x,r)\subseteq B_0$, $r\in (0,R]$.
For each $j\in\mathbb{N}$, 
define $r_j$, $u_j$, $g_j$, and $B^j$ as in Construction~\ref{Cons1}.
Since $u_j\in M^{1,p}(\sigma B_0,d,\mu)$,
for each $j\in\bbbn$ we have
\begin{equation}
\label{HD-1}
C_S R\left(\, \mvint_{\sigma B_0} g_j^p\, d\mu\right)^{\!\!\!1/p}
=\frac{C_S R 2^{j+2}}{r}\left(\frac{\mu(B^j)}{\mu(\sigma B_0)}\right)^{\!\!1/p},
\end{equation}
and
\begin{equation}
\label{HD-2}
C_S\left(\, \mvint_{\sigma B_0}|u_j|^p\, d\mu\right)^{\!\!\!1/p}
\leq C_S\left(\frac{\mu(B^j)}{\mu(\sigma B_0)}\right)^{\!\!1/p}.
\end{equation}
Moreover, 
\begin{equation}
\label{HD-3}
\left(\, \mvint_{B_0} |u_j|^{p^*}\, d\mu\right)^{\!\!\!1/p^*}
\geq\left(\frac{\mu(B^{j+1})}{\mu(B_0)}\right)^{\!\!1/p^*}.
\end{equation}
In concert, \eqref{HD-1}-\eqref{HD-3}, and \eqref{eq-JK2}, 
give
$$
\left(\frac{\mu(B^{j+1})}{\mu(B_0)}\right)^{\!\!1/p^*}
\leq 
C_S\bigg(\frac{R\, 2^{j+2}}{r}+1\bigg)
\left(\frac{\mu(B^j)}{\mu(\sigma B_0)}\right)^{\!\!1/p}
\leq
\frac{C_S\, R\, 2^{j+3}}{r}\left(\frac{\mu(B^j)}{\mu(B_0)}\right)^{\!\!1/p},
$$
where the last inequality follows from the observation that $r\leq R$ and $\mu(B_0)\leq\mu(\sigma B_0)$.
Therefore,
$$
\mu(B^{j+1})^{1/p^*}
\leq
\left(\frac{8 C_SR}{r\mu(B_0)^{1/s}}\right)\,2^j\,\mu(B^j)^{1/p},
\quad\forall\,j\in\mathbb{N}.
$$
Applying Lemma~\ref{iteration} with
$$
a_j:=\mu(B^j),\quad 
p:=p,\quad 
q:=p^*,\quad 
\rho:=\frac{8\,C_S\, R}{r\mu(B_0)^{1/s}}
\quad
\text{and}
\quad
\tau:=2,
$$
and remembering that $a_1=\mu(B^1)\leq\mu(B)$ yields
$$
1\leq \mu(B)^{1-p/p^*}\left(\frac{8\,C_S\, R}{r\mu(B_0)^{1/s}}\right)^p\, 2^{pp^*/(p^*-p)}=
\left(\frac{\mu(B)}{\mu(B_0)}\right)^{p/s}(8C_S)^p\left(\frac{R}{r}\right)^p\, 2^s.
$$
Therefore,
$$
\frac{\mu(B(x,r))}{\mu(B(y,R))}\geq
(8C_S)^{-s}\, 2^{-s^2/p}\, \left(\frac{r}{R}\right)^s,
$$
whenever $B(x,r)\subseteq B(y,R)$, $r\in (0,R]$. This completes the proof of \eqref{measboundthm} with $\kappa =(8C_S)^{-s}\, 2^{-s^2/p}$.

There remains to prove that {\it (c)} implies {\it (a)}
under the additional assumption that $(X,d)$ is uniformly perfect. 
Recall that we may assume that $0<\lambda<1/5$, where $\lambda$ as as in \eqref{U-perf}.
Let $x,y\in X$, $B:=B(x,r)\subseteq B(y,R)$, $0<r\leq R<\infty$. In light of part {\it (ii)} of Lemma~\ref{en2-4}, we may assume that $r\leq 3\varphi_x(r)/\lambda^2$. 
Let $\widetilde{r}_j$, $\widetilde{u}_j$, $\widetilde{g}_j$, and $\widetilde{B}^j$, $j\in\bbbn$ be as in Construction~\ref{Cons2}.
Then, the functions $\widetilde{u}_j$ and $\widetilde{g}_j$ satisfy \eqref{eq-JK-D} (used here with $B_0:=B(y,R)$).
Observe that for each $j\in\mathbb{N}$, we have
(keeping in mind that $\sigma>1$)
\begin{equation}
\label{Gkw-1}
C_P R\left(\,\mvint_{\sigma B_0} \widetilde{g_j}^p\, d\mu\right)^{1/p}=
\frac{C_P\, R\, 2^{j+2}}{\varphi_x(r)}\left(\frac{\mu(\widetilde{B}^j)}{\mu(\sigma B_0)}\right)^{1/p}\leq
\frac{3\, C_P\, R\, 2^{j+2}}{\lambda^2r}
\left(\frac{\mu(\widetilde{B}^j)}{\mu(B_0)}\right)^{1/p}.
\end{equation}
Now \eqref{Gkw-1} combined with \eqref{eq-JK-D} and \eqref{eq51} yields
$$
\frac{1}{2}\left(\frac{\mu(\widetilde{B}^{j+1})}{\mu(B_0)}\right)^{1/p^*}
\leq\frac{3\, C_P\, R\, 2^{j+2}}{\lambda^2 r}
\left(\frac{\mu(\widetilde{B}^{j})}{\mu(B_0)}\right)^{1/p}.
$$
Therefore,
$$
\mu(\widetilde{B}^{j+1})^{1/p^*}
\leq \left(\frac{24\, C_P\, R}{\lambda^2 r\,\mu(B_0)^{1/s}}\right)\, 2^j\, \mu(\widetilde{B}^j).
$$
Applying Lemma~\ref{iteration} with
$$
a_j:=\mu(\widetilde{B}^j),\quad
p:=p,\quad
q:=p^*,\quad
\rho:=\frac{24\, C_P\, R}{\lambda^2\, r\, \mu(B_0)^{1/s}}
\quad
\text{and}
\quad
\tau:=2,
$$
and keeping in mind that
$a_1=\mu(\widetilde{B}^1)\leq \mu(B)=\mu(B(x,r))$, yields
$$
1\leq\mu(B)^{1-p/p^*}\,
\left(\frac{24\, C_P\, R}{\lambda^2\, r\, \mu(B_0)^{1/s}}\right)^p\,
2^{pp^*/(p^*-p)}=
\left(\frac{\mu(B)}{\mu(B_0)}\right)^{p/s}
\left(\frac{R}{r}\right)^p
\left(\frac{24\, C_P}{\lambda^2}\right)^p\, 2^s.
$$
Therefore,
$$
\frac{\mu(B(x,r))}{\mu(B(y,R))}\geq
\left(\frac{\lambda^2}{24\, C_P}\right)^s\, 2^{-s^2/p}\,
\left(\frac{r}{R}\right)^s.
$$
This finishes the proof of the theorem.
\end{proof}

In the next two results we investigate the relationship between the lower measure bound
in \eqref{eq22} and global Sobolev and Sobolev-Poincar\'e inequalities, i.e., estimates of the form
\begin{equation}
\label{yvx-X1}
\left(\, \int_{X} |u|^{p^*}\, d\mu\right)^{\!\!\!1/p^*}
\leq C_S\left(\, \int_{X} g^{p}\, d\mu\right)^{\!\!\!1/p}+C_S\left(\, \int_{X} |u|^{p}\, d\mu\right)^{\!\!\!1/p},
\end{equation}
and
\begin{equation*}
\inf_{\gamma\in\bbbr}\left(\, \int_{X} |u-\gamma|^{p^*}\, d\mu\right)^{\!\!\!1/p^*}\leq
C_P\left(\, \int_{X} g^p\, d\mu\right)^{\!\!\!1/p},
\end{equation*}
where $u\in M^{1,p}(X,d,\mu)$ and $g\in D(u)$. It was shown in \cite{gorka} that if the
measure $\mu$ is doubling then the Sobolev embedding in \eqref{yvx-X1} implies the measure
$\mu$ satisfies the lower bound in \eqref{eq22}. In Theorem~\ref{GlobalEmbedd1}, we prove
that the assumption $\mu$ is doubling is not necessary.

\begin{theorem}\label{GlobalEmbedd1}
Suppose $(X,d,\mu)$ is a metric measure space. Fix $s,p\in(0,\infty)$ such that
$p<s$ and set $p^*:=sp/(s-p)$. Then the following statements are valid.
\begin{enumerate}[(a)]
\item If there exists a finite constant $C_S>0$ satisfying
\begin{equation*}
\left(\, \int_{X} |u|^{p^*}\, d\mu\right)^{\!\!\!1/p^*}
\leq C_S\left(\, \int_{X} g^{p}\, d\mu\right)^{\!\!\!1/p}+C_S\left(\, \int_{X} |u|^{p}\, d\mu\right)^{\!\!\!1/p},
\end{equation*}
whenever $u\in M^{1,p}(X,d,\mu)$ and $g\in D(u)$, then 
there exists a finite constant $\kappa>0$ such that
\begin{equation*}
\mu\big(B(x,r)\big)\geq \kappa\,r^s\quad\mbox{for every $x\in X$ and for every\, $r\in(0,1]$.}
\end{equation*}

\item If the metric space $(X,d)$ is uniformly perfect and there exists a finite
constant $C_P>0$ satisfying
\begin{equation}
\label{yvx-1}
\inf_{\gamma\in\bbbr}\left(\, \int_{X} |u-\gamma|^{p^*}\, d\mu\right)^{\!\!\!1/p^*}\leq
C_P\left(\, \int_{X} g^p\, d\mu\right)^{\!\!\!1/p},
\end{equation}
whenever $u\in M^{1,p}(X,d,\mu)$ and $g\in D(u)$, then there exists a finite constant $\kappa>0$ such that
\begin{equation*}
\mu\big(B(x,r)\big)\geq \kappa\,r^s\quad\mbox{for every $x\in X$ and for every finite\, $r\in\big(0,{\rm diam}(X)\big]$.}
\end{equation*}
\end{enumerate}
\end{theorem}
\begin{remark}
As a consequence, the global Sobolev-Poincar\'e inequality in \eqref{yvx-1} implies the local Sobolev-Poincar\'e inequality in \eqref{HHs-175} of Theorem~\ref{PlessS}. In fact, within the class of $M^{1,p}$-spaces on uniformly perfect metric measure spaces, the global Sobolev-Poincar\'e inequality in \eqref{yvx-1} implies all of the local estimates in \eqref{hdx-1}-\eqref{hdx-3} as well as the global H\"older condition in \eqref{hdx-4}.
\end{remark}

The proofs of parts {\em (a)} and {\em (b)} are similar to the proofs of the implications 
{\em (b)} $\Rightarrow$ {\em (a)} and {\em (c)} $\Rightarrow$ {\em (a)} in Theorem~\ref{PlessS}, respectively. We leave details to the reader.

\begin{corollary}\label{GlobalEmbeddCor}
Let $(X,d,\mu)$ be a bounded uniformly perfect metric measure space, 
and fix $s,p\in(0,\infty)$ such that $p<s$. Then with $p^*:=sp/(s-p)$, the following statements are equivalent. 
\begin{enumerate}[(a)]
\item There exists a finite constant $\kappa>0$ such that
\begin{equation*}
\mu\big(B(x,r)\big)\geq\kappa\,r^s\,\,\,\mbox{ for every }\,x\in X
\mbox{ and every finite }r\in\big(0,{\rm diam}(X)\big].
\end{equation*}

\item One has $M^{1,p}(X,d,\mu)\subseteq L^{p^*}(X,\mu)$ and 
there exists a finite constant $C_S>0$ satisfying
\begin{equation}
\label{GBLS}
\|u\|_{L^{p^*}(X,\mu)}\leq C_S\|u\|_{M^{1,p}(X,d,\mu)},\quad\forall\,u\in M^{1,p}(X,d,\mu).
\end{equation}

\item There exists a finite constant $C_P>0$ satisfying
\begin{equation}\label{GBLP}
\inf_{\gamma\in\bbbr}\left(\, \int_{X} |u-\gamma|^{p^*}\, d\mu\right)^{\!\!\!1/p^*}\leq
C_P \left(\, \int_{X} g^p\, d\mu\right)^{\!\!\!1/p},
\end{equation}
whenever $u\in M^{1,p}(X,d,\mu)$ and $g\in D(u)$.
\end{enumerate}

Consequently, in the context of bounded uniformly perfect metric measure spaces,
the global estimates in \eqref{GBLS}-\eqref{GBLP} are equivalent to the local estimates
in \eqref{hdx-1}-\eqref{hdx-3} as well as the global H\"older condition in \eqref{hdx-4}.
\end{corollary}

\begin{proof}
It is clear from Theorem~\ref{GlobalEmbedd1} that {\em (c)} implies {\em (a)}, and that {\em (b)} implies that
that there exists $\kappa\in(0,\infty)$ such that
\begin{equation*}
\mu\big(B(x,r)\big)\geq \kappa\,r^s \quad \text{for every $x\in X$.
and every $r\in(0,1]$.}
\end{equation*}
If however,
$1<r\leq{\rm diam}(X)<\infty$, then $[{\rm diam}(X)]^{-1}r\leq1$, which implies
$$\mu\big(B(x,r)\big)\geq \mu\big(B\big(x,[{\rm diam}(X)]^{-1}r\big)\big)\geq\kappa[{\rm diam}(X)]^{-s}\,r^s.$$
Hence, the statement in {\it (a)} holds.

On the other hand, given that ${\rm diam}(X)<\infty$, we can
find a constant  $\widetilde{\kappa}\in(0,\infty)$ (depending only on $\kappa$, $s$, 
and the space $X$) so that $\mu$ satisfies the $V(2B_0,s,\widetilde{\kappa})$ condition (cf.~\eqref{measbound}), where $B_0=X$ is any ball of radius $R_0:=2{\rm diam}(X)$.
Consequently, \eqref{GBLS} and \eqref{GBLP} now follow immediately from \eqref{eq18} and \eqref{eq19} in Theorem~\ref{embedding}. This finishes the proof of the corollary.
\end{proof}

\section{The Case $p=s$}
\label{S5}

\begin{theorem}\label{AS-U-X}
Suppose that $(X,d,\mu)$ is a uniformly perfect measure
metric space. Then for each fixed $s\in(0,\infty)$ and
$\sigma\in(1,\infty)$,
the following two statements are equivalent.
\begin{enumerate}[(a)]
\item There exists a constant $\kappa\in (0,\infty)$ such that
\begin{equation}
\label{Gy-21-X}
\mu\big(B(x,r)\big)\geq \kappa\,r^s
\quad
\text{for every $x\in X$ and every finite $r\in\big(0,{\rm diam}(X)\big].$}
\end{equation}

\item There exist constants $C_1,C_2,\gamma\in(0,\infty)$ such that
\begin{equation}\label{Iue.5-X}
\mvint_{B_0} {\rm exp}\bigg(C_1\frac{|u-u_{B_0}|}{\|g\|_{L^{s}(\sigma B_0,\mu)}}\bigg)^\gamma\,d\mu\leq C_2,
\end{equation}
whenever $B_0\subseteq X$ is a ball with radius at
most ${\rm diam}(X)$, $u\in M^{1,s}(\sigma B_0,d,\mu)$ and $g\in D_+(u)$.
\end{enumerate}
\end{theorem}

\begin{remark}
The implication {\it (a)} $\Rightarrow$ {\it (b)} holds in any metric-measure space, not necessarily uniformly perfect. 
\hfill$\blacksquare$
\end{remark}

\begin{proof}
Fix $s\in(0,\infty)$ and $\sigma\in(1,\infty)$. 
For the implication {\it (a)} $\Rightarrow$ {\it (b)}, observe that if $B_0$ is a ball having finite radius $R_0\in\big(0,{\rm diam}(X)\big]$, then \eqref{Gy-21-X} implies that $\mu$ satisfies the
$V(\sigma B_0,s,b)$ condition
with $b:=\kappa\sigma^{-s}\in(0,\infty)$, see \eqref{eq52}.
As such, the desired conclusion now follows from part {\it (b)} in Theorem~\ref{embedding} with $\gamma=1$.

Regarding the opposite implication, suppose that
\eqref{Iue.5-X} holds for some $C_1,C_2,\gamma\in(0,\infty)$, 
and fix $x\in X$, $r\in\big(0,{\rm diam}(X)\big]$, finite. 
Let $B=B(x,r)$.
According to part {\it (i)} of Lemma~\ref{en2-4}, it suffices to prove \eqref{Gy-21-X} under the additional assumption that $r\leq 3\varphi_x(r)/\lambda^2$, where $0<\lambda<1/5$. Then 
$0<\varphi_x(r)<r$ and we define $\widetilde{r}_j$, $\widetilde{u}_j$, $\widetilde{g}_j$, and $\widetilde{B}^j$, $j\in\bbbn$ as in Construction~\ref{Cons2}.
In particular, we have that $\widetilde{u}_j$ and $\widetilde{g}_j$ satisfy \eqref{Iue.5-X} (used here with $B_0:=B$), i.e.,
\begin{equation}
\label{UEn.11}
\mvint_{B} {\rm exp}\bigg(C_1\frac{|\widetilde{u}_j-(\widetilde{u}_j)_{B}|}{\|\widetilde{g}_j\|_{L^{s}(\sigma B,\mu)}}\bigg)^\gamma\,d\mu\leq C_2,
\qquad\forall\,j\in\mathbb{N}.
\end{equation}
Since $\sigma>1$ and $r\leq 3\varphi_x(r)/\lambda^2$, we have
\begin{equation}
\label{yre-1}
\|\widetilde{g}_j\|_{L^s(\sigma B,\mu)}=\frac{2^{j+2}}{\varphi_x(r)}\,\mu(\widetilde{B}^j)^{1/s}
\leq 3\,\frac{2^{j+2}}{\lambda^2\,r}\,\mu(\widetilde{B}^j)^{1/s}.
\end{equation}
Now \eqref{eq51},  \eqref{UEn.11} and  \eqref{yre-1} give
\begin{align}
\label{Iue.5-2-X}
\frac{\mu(\widetilde{B}^{j+1})}{\mu(B)}\,
{\rm exp}\bigg(\frac{C_1\lambda^2\,r}{24\cdot 2^{j}\,\mu(\widetilde{B}^{j})^{1/s}}\bigg)^\gamma\leq C_2.
\end{align}
Without loss of generality we can assume that $C_2>1$. Then it
follows from \eqref{Iue.5-2-X} that
\begin{align*}
\frac{C_1\lambda^2\,r}{24\cdot 2^{j}\,\mu(\widetilde{B}^{j})^{1/s}}
&\leq\bigg[\log\bigg(C_2\frac{\mu(B)}{\mu(\widetilde{B}^{j+1})}\bigg)\bigg]^{1/\gamma}
\leq (2s\gamma^{-1})^{1/\gamma}\,C_2^{1/(2s)}\bigg(\frac{\mu(B)}{\mu(\widetilde{B}^{j+1})}\bigg)
^{1/(2s)},
\end{align*}
where the last inequality
follows from the estimate
$\log(y)\leq q\,y^{1/q}$, which holds true for all $y,q\in(0,\infty)$
(applied here with $q=2s\gamma^{-1}$).
Therefore,
\begin{equation*}
\mu(\widetilde{B}^{j+1})^{1/(2s)}\leq 
\left(\frac{24(2s\gamma^{-1})^{1/\gamma}\,C_2^{1/(2s)}\mu(B)^{1/(2s)}}{C_1\lambda^2\,r}\right)\,2^{j}\,\mu(\widetilde{B}^{j})^{1/s}.
\end{equation*}
Now applying Lemma~\ref{iteration} with
$$
a_j:=\mu\big(\widetilde{B}^j\big),\quad
p:=s,\quad
q:=2s,\quad
\rho:=\frac{24(2s\gamma^{-1})^{1/\gamma}\,C_2^{1/(2s)}\mu(B)^{1/(2s)}}{C_1\lambda^2\,r},
\quad
\text{and}
\quad
\tau:=2,
$$
we get
$$
1\leq
\mu\big(\widetilde{B}^1\big)^{1-\frac{s}{2s}}\,
\left(\frac{24(2s\gamma^{-1})^{1/\gamma}\,C_2^{1/(2s)}\mu(B)^{1/(2s)}}{C_1\lambda^2\,r}\right)^s\,
2^{\frac{s\cdot 2s}{2s-s}}
\leq 
\mu(B) 
\frac{24^s(2s\gamma^{-1})^{s/\gamma}\,C_2^{1/2}}{C_1^s\lambda^{2s}\,r^s}\, 4^s
$$
so
$$
\mu(B(x,r))\geq
\bigg(\frac{C_1^s\,\lambda^{2s}}{96^s(2s\gamma^{-1})^{s/\gamma}\sqrt{C_2}}\bigg)
\,r^s.
$$
The proof is complete.
\end{proof}

In the case of doubling measures we have the following
characterization which is a consequence of 
Theorem~\ref{embedding} and Theorem~\ref{AS-U-BDD2}, below.

\begin{theorem}\label{AS-U-BDD}
Suppose that $(X,d,\mu)$ is a uniformly perfect measure
metric space and fix $\sigma\in(1,\infty)$. Then the following two statements are equivalent.
\begin{enumerate}[(a)]
\item The measure $\mu$ is doubling.
\vskip.08in

\item There exist constants $C_1,C_2,s,\gamma\in(0,\infty)$ such that
\begin{equation}
\label{eq53}
\mvint_{B_0} {\rm exp}\bigg(C_1\frac{\mu(\sigma B_0)^{1/s}|u-u_{B_0}|}{R_0\|g\|_{L^s(\sigma B_0,\mu)}}\bigg)^\gamma\,d\mu\leq C_2,
\end{equation}
whenever $B_0\subseteq X$ is a ball, $u\in M^{1,s}(\sigma B_0,d,\mu)$ and $g\in D_+(u)$. 
\end{enumerate}
\end{theorem}

For the implication {\it (a)} $\Rightarrow$ {\it (b)} in Theorem~\ref{AS-U-BDD}, one can take $s:=\log_2(C_\mu)$ where
\begin{equation*}
C_\mu:=\sup_{x\in X,\, r\in(0,\infty)}\frac{\mu(B(x,2r))}{\mu(B(x,r))}\in(1,\infty)
\end{equation*}
is the doubling constant for $\mu$, see \eqref{Doub-2}. Then $\mu$ satisfies the $V(\sigma B_0,s,b)$ condition with 
$b=\kappa\mu(\sigma B_0)(\sigma R_0)^{-s}$ and \eqref{eq53} follows from Theorem~\ref{embedding}.

\begin{theorem}\label{AS-U-BDD2}
Let $(X,d,\mu)$ be a uniformly perfect measure
metric space and suppose that there exist 
$s,C_1,C_2,\gamma\in(0,\infty)$ and $\sigma\in[1,\infty)$ such that
\begin{equation}\label{Iue.5-52}
\mvint_{B_0} {\rm exp}\bigg(C_1\frac{\mu(\sigma B_0)^{1/s}|u-u_{B_0}|}{R_0\|g\|_{L^s(\sigma B_0,\mu)}}\bigg)^\gamma\,d\mu\leq C_2,
\end{equation}
whenever $B_0\subseteq X$ is a ball, $u\in M^{1,s}(\sigma B_0,d,\mu)$ and $g\in D_+(u)$.

Then for every $\varepsilon\in(0,\infty)$, there exists a constant $\kappa\in(0,\infty)$ such
that 
\begin{equation}\label{measboundthm-52}
\frac{\mu\big(B(x,r)\big)}{\mu\big(B(y,R)\big)}\geq \kappa\bigg(\frac{r}{R}\bigg)^{s+\varepsilon},
\end{equation}
whenever $x,y\in X$, $B(x,r)\subseteq B(y,R)$, $0<r\leq R<\infty$. In particular, the measure $\mu$ is doubling.
\end{theorem}

\begin{proof}
Note that the estimate in \eqref{measboundthm-52} will follow once we prove that for each  $\beta\in(1,\infty)$, there exists $\kappa_\beta\in(0,\infty)$ satisfying
\begin{equation}\label{ME13-52}
\frac{\mu\big(B(x,r)\big)}{\mu\big(B(y,R)\big)}\geq \kappa_\beta\bigg(\frac{r}{R}\bigg)^{\beta s/(\beta-1)},
\end{equation}
whenever $x,y\in X$ and $B(x,r)\subseteq B(y,R)$, $0<r\leq R<\infty$. To this end, fix $\beta\in(1,\infty)$ and suppose  $B:=B(x,r)\subseteq B(y,R)$ for some $x,y\in X$ and $0<r\leq R<\infty$. Recall that we may assume that $0<\lambda<1/5$, where $\lambda$ as as in \eqref{U-perf}. As such, appealing to {\it (ii)} of Lemma~\ref{en2-4}, it suffices to only consider the case when $r\leq 3\varphi_x(r)/\lambda^2$. 
Next, let $\widetilde{r}_j$, $\widetilde{u}_j$, $\widetilde{g}_j$, and $\widetilde{B}^j$, $j\in\bbbn$ be as in Construction~\ref{Cons2}.
Then, it follows \eqref{Iue.5-52} (applied here with $B_0:=B(y,R)$) that the functions $\widetilde{u}_j$ and $\widetilde{g}_j$ satisfy 
\begin{equation}\label{beiu-2}
\mvint_{B_0} {\rm exp}\bigg(C_1\frac{\mu(\sigma B_0)^{1/s}|\widetilde{u}_j-(\widetilde{u}_j)_{B_0}|}{R\|\widetilde{g}_j\|_{L^s(\sigma B_0,\mu)}}\bigg)^\gamma\,d\mu\leq C_2,\quad\forall\,j\in\mathbb{N}.
\end{equation}

Observe that for each $j\in\mathbb{N}$, we have
\begin{equation}
\label{byt-1}
R\left(\,\mvint_{\sigma B_0} \big(\widetilde{g_j}\big)^s\, d\mu\right)^{1/s}=
\frac{R\, 2^{j+2}}{\varphi_x(r)}\left(\frac{\mu(\widetilde{B}^j)}{\mu(\sigma B_0)}\right)^{1/s}\leq
\frac{3\, R\, 2^{j+2}}{\lambda^2r}
\left(\frac{\mu(\widetilde{B}^j)}{\mu(B_0)}\right)^{1/s},
\end{equation}
where, in obtaining the inequality in \eqref{byt-1}, we have used the fact that $r\leq 3\varphi_x(r)/\lambda^2$.
Altogether, \eqref{byt-1}, \eqref{beiu-2}, and \eqref{eq51} yield
\begin{align}
\label{fxz-1}
\frac{\mu(\widetilde{B}^{j+1})}{\mu(B_0)}\,
{\rm exp}\bigg(\frac{C_1\lambda^2\,r\,\mu(B_0)^{1/s}}{24\cdot 2^{j}R\,\mu(\widetilde{B}^{j})^{1/s}}\bigg)^\gamma\leq C_2
\end{align}
Note we can assume without loss of generality that $C_2>1$. Then using the estimate $\log(y)\leq q\,y^{1/q}$ (which holds true for all $y,q\in(0,\infty)$) with $q=\beta s\gamma^{-1}$, we may conclude from \eqref{fxz-1} that for each $j\in\mathbb{N}$, there holds
\begin{align*}
\frac{C_1\lambda^2\,r}{24\cdot 2^{j}R}
\cdot\frac{\mu(B_0)^{1/s}}{\mu(\widetilde{B}^{j})^{1/s}}
&\leq\bigg[\log\bigg(C_2\frac{\mu( B_0)}{\mu(\widetilde{B}^{j+1})}\bigg)\bigg]^{1/\gamma}
\leq (\beta s\gamma^{-1})^{1/\gamma}\,C_2^{1/(\beta s)}\bigg(\frac{\mu(B_0)}{\mu(\widetilde{B}^{j+1})}\bigg)
^{1/(\beta s)}.
\end{align*}
Therefore,
\begin{align*}
\mu(\widetilde{B}^{j+1})^{1/(\beta s)}\leq\left(
\frac{24R(\beta s\gamma^{-1})^{1/\gamma}\,C_2^{1/(\beta s)}}{C_1\lambda^2\,r\mu(B_0)^{(\beta-1)/(\beta s)}}\right)\,2^{j}\,\mu(\widetilde{B}^{j})^{1/s}.
\end{align*}

Now, invoking Lemma~\ref{iteration} with
$$
a_j:=\mu(\widetilde{B}^j),\quad
p:=s,\quad
q:=\beta s,\quad
\rho:=\frac{24R(\beta s\gamma^{-1})^{1/\gamma}\,C_2^{1/(\beta s)}}{C_1\lambda^2\,r\mu(B_0)^{(\beta-1)/(\beta s)}}
\quad
\text{and}
\quad
\tau:=2,
$$
and keeping in mind that
$a_1=\mu(\widetilde{B}^1)\leq \mu(B)=\mu(B(x,r))$ and $B_0=B(y,R)$, yields
\begin{equation*}
\begin{split}
1&\leq\mu(B)^{1-1/\beta}\,
\left(\frac{24R(\beta s\gamma^{-1})^{1/\gamma}\,C_2^{1/(\beta s)}}{C_1\lambda^2\,r\mu(B_0)^{(\beta-1)/(\beta s)}}\right)^s\,
2^{\beta s^2/(\beta s-s)}\\
&=
\left(\frac{\mu(B(x,r))}{\mu(B(y,R))}\right)^{(\beta-1)/\beta}
\left(\frac{R}{r}\right)^s
\left(\frac{24(\beta s\gamma^{-1})^{1/\gamma}\,C_2^{1/(\beta s)}}{C_1\lambda^2}\right)^s\, 2^{\beta s/(\beta-1)}.
\end{split}
\end{equation*}
Therefore, 
$$
\frac{\mu(B(x,r))}{\mu(B(y,R))}\geq
\left(\frac{C_1\lambda^2}{24(\beta s\gamma^{-1})^{1/\gamma}\,C_2^{1/(\beta s)}}\right)^s\, 2^{-\beta^2 s/(\beta-1)^2}\,
\left(\frac{r}{R}\right)^{\beta s/(\beta-1)}.
$$
This finishes the proof \eqref{ME13-52} and, in turn, the proof of the theorem.
\end{proof}

\section{The Case $p>s$}
\label{S6}

\begin{theorem}\label{CX-03}
Suppose that $(X,d,\mu)$ is a uniformly perfect metric measure space and fix $s,p\in(0,\infty)$ satisfying
$p>s$. Then the following two statements are equivalent. 
\begin{enumerate}[(a)]
\item There exists a finite constant $\kappa>0$ such that
\begin{equation}
\label{Ge8-31}
\mu\big(B(x,r)\big)\geq \kappa\,r^s\,\,\,\mbox{for every }\,x\in X
\mbox{ and every finite }r\in\big(0,{\rm diam}(X)\big].
\end{equation}

\item There exists a constant $C_H\in(0,\infty)$ 
with the property that for each $u\in M^{1,p}(X,d,\mu)$ and $g\in D(u)$, there holds
\begin{equation}\label{Knv-874}
|u(x)-u(y)|\leq C_Hd(x,y)^{1-s/p}\|g\|_{L^{p}(X,\mu)},
\qquad\forall\,x,y\in X,
\end{equation}
Hence, every function $u\in M^{1,p}(X,d,\mu)$
has H\"older continuous representative of order $(1-s/p)$
on $X$.
\end{enumerate}
\end{theorem}

\begin{proof}
We begin by proving the implication {\it (a)} $\Rightarrow$ {\it (b)}. Fix any ball $B_0$ having finite radius $R_0\in\big(0,2\,{\rm diam}(X)\big]$. 
If $B(x,r)\subseteq 2B_0$ and $r\in (0,2R_0]$, then $4^{-1}r\leq{\rm diam}(X)$, and \eqref{Ge8-31} gives
\begin{equation*}
\mu(B(x,r))\geq\mu(B(x,4^{-1}r))\geq\kappa(4^{-1}r)^s=\widetilde{\kappa}r^s,
\end{equation*}
where $\widetilde{\kappa}=\kappa 4^{-s}$. Thus, $\mu$ satisfies the $V(2B_0,s,\widetilde{\kappa})$ condition (see \eqref{measbound}).
As such, part {\it (c)} in Theorem~\ref{embedding} guarantees the existence of a constant $C\in(0,\infty)$ (independent of $B_0$) with the property that whenever $u\in M^{1,p}(2B_0,d,\mu)$ and $g\in D(u)$, there holds
\begin{equation}\label{Jnb.1}
|u(x)-u(y)|\leq C\,\widetilde{\kappa}^{-1/p}d(x,y)^{1-s/p}\left(\, \int_{2B_0} g^p\, d\mu\right)^{\!\!\!1/p}
\quad
\text{for all $x,y\in B_0$.}
\end{equation}
Now, if  $u\in M^{1,p}(X,d,\mu)$ and $g\in D(u)$, then their pointwise restrictions to the ball $2B_0$ (still denoted by  $u$ and $g$) continue to satisfy \eqref{eq21}. Hence, this restriction of $u$ belongs to $M^{1,p}(2B_0,d,\mu)$, and $g$ remains a generalized gradient of $u$. Therefore, by \eqref{Jnb.1} we have
\begin{equation}\label{Jnb.10}
|u(x)-u(y)|\leq C\,\widetilde{\kappa}^{-1/p}d(x,y)^{1-s/p}\|g\|_{L^{p}(X,\mu)}
\quad
\text{for all $x,y\in B_0$.}
\end{equation}
Given that the constants $C$ and $\widetilde{\kappa}$
are independent of the arbitrarily chosen $B_0$, it follows 
that \eqref{Jnb.10} implies \eqref{Knv-874}, finishing
the proof of {\it (a)} $\Rightarrow$ {\it (b)}.
 
For the reverse implication, fix $x\in X$ and $r\in\big(0,{\rm diam}(X)\big]$, finite. If $B(x,r)=X$ then $\diam(X)\in(0,\infty)$ and 
\begin{equation*}
\mu\big(B(x,r)\big)=\mu(X)\geq \mu(X)[{\rm diam}(X)]^{-s}\,r^s.
\end{equation*}
Thus, in what follows, 
we may assume that $X\setminus B(x,r)\neq\emptyset$.

Let $\lambda\in(0,1)$ be as in \eqref{U-perf} and define 
$u:X\to[0,1]$ by setting $u:=\Phi_{0,\lambda r}$,
where the function $\Phi_{0,\lambda r}$ is as in Lemma~\ref{Lipbump}.
Then, $u\in M^{1,p}(X,d,\mu)$ and the function $g:=(\lambda r)^{-1}\chi_{B(x,\lambda r)}$ belongs to $D(u)$. Moreover, since $(X,d)$ is assumed to be
uniformly perfect, we may select a point 
$y\in B(x,r)\setminus B(x,\lambda r)$.
Then by \eqref{Knv-874} (used here
with this choice of $u$ and $g$), we have
$$
1=|u(x)-u(y)|\leq C_Hd(x,y)^{1-s/p}\|g\|_{L^{p}(X,\mu)}
\leq C_H\lambda^{-1}r^{-s/p}\mu\big(B(x,r)\big)^{1/p},
$$
from which \eqref{Ge8-31} follows with 
$\kappa:=(\lambda/C_H)^p\in(0,\infty)$. This finishes the proof
of the reverse implication and, in turn, the proof of the
theorem.
\end{proof}

\begin{theorem}\label{CX-032}
Suppose that $(X,d,\mu)$ is a uniformly perfect metric measure space and fix $s\in(0,\infty)$, $\sigma\in(1,\infty)$. Then for 
each $p\in(s,\infty)$, the following two statements are equivalent. 
\begin{enumerate}[(a)]
\item There exists a constant $\kappa\in(0,\infty)$ satisfying
\begin{equation}
\label{measboundthm-tt2}
\frac{\mu\big(B(x,r)\big)}{\mu\big(B(y,R)\big)}\geq \kappa\bigg(\frac{r}{R}\bigg)^{\!\!s},
\end{equation}
whenever $x,y\in X$ satisfy $B(x,r)\subseteq B(y,R)$ and $0<r\leq R<\infty$.

\item There exists a finite constant $C_H>0$ 
such that for each ball $B_0:=B(x_0,R_0)$ with $x_0\in X$ and $R_0\in(0,\infty)$, and each $u\in M^{1,p}(\sigma B_0,d,\mu)$ and $g\in D(u)$, there holds
\begin{equation}\label{XBP-82}
|u(x)-u(y)|\leq C_Hd(x,y)^{1-s/p}R_0^{s/p}\left(\,\mvint_{\sigma B_0}g^p\,d\mu\right)^{\!\!\!1/p},
\quad\forall\,x,y\in B_0,
\end{equation}
Hence, every function $u\in M^{1,p}(\sigma B_0,d,\mu)$
has H\"older continuous representative of order $(1-s/p)$
on $B_0$.
\end{enumerate}
\end{theorem}

\begin{proof}
We begin proving the implication {\it (a)} $\Rightarrow$ {\it (b)}. Given a ball $B_0:=B(x_0,R_0)$ with $x_0\in X$ and 
$R_0\in(0,\infty)$, let $B(y,R):=\sigma B_0$. 
Then inequality 
\eqref{measboundthm-tt2} implies that the measure $\mu$ satisfies the $V(\sigma B_0,s,b)$ condition with 
$b=\kappa\mu(\sigma B_0)(\sigma R_0)^{-s}$.
As such, for this value of $b$ the inequality displayed in \eqref{XBP-82} follows immediately from 
\eqref{eq30} in Theorem~\ref{embedding}.

In order to prove the implication {\it (b)} implies {\it (a)}, fix points $x,y\in X$ and suppose that $B:=B(x,r)\subseteq B(y,R)$, where $0<r\leq R<\infty$. Specializing \eqref{XBP-82} to the case when $B_0:=B(y,R)$ implies
that
\begin{equation}\label{XBP-83}
|u(z)-u(w)|\leq C_H\,d(z,w)^{1-s/p}R^{s/p}\left(\,\mvint_{\sigma B_0}g^p\,d\mu\right)^{\!\!\!1/p},
\quad\forall\,z,w\in B(y,R),
\end{equation}
whenever $u\in M^{1,p}(\sigma B_0,d,\mu)$ and $g\in D(u)$.
If $B(x,r)=X$ then $B(y,R)=X$ given that $B(x,r)\subseteq B(y,R)$. Thus, in this
case, \eqref{measboundthm-tt2} trivially holds with
any $\kappa\in(0,1]$. As such, in what follows we will assume that $X\setminus B(x,r)\neq\emptyset$.

Let $\lambda\in(0,1)$ be as in \eqref{U-perf} and define 
$u:X\to[0,1]$ by setting $u:=\Phi_{0,\lambda r}$,
where the function $\Phi_{0,\lambda r}$ is as in Lemma~\ref{Lipbump}.
Then, $u\in M^{1,p}(\sigma B_0,d,\mu)$ and the function 
$g:=(\lambda r)^{-1}\chi_{B(x,\lambda r)}$ belongs to $D(u)$. 
Since $(X,d)$ is assumed to be
uniformly perfect, we may select a point 
$w\in B(x,r)\setminus B(x,\lambda r)$.
Then by \eqref{XBP-83}, we have
(keeping in mind $\sigma>1$)
\begin{align}
1=|u(x)-u(w)|&\leq  C_H\,d(x,w)^{1-s/p}R^{s/p}\left(\,\mvint_{\sigma B_0}g^p\,d\mu\right)^{\!\!\!1/p}
\nonumber\\
&\leq C_H\,\lambda^{-1}\bigg(\frac{R}{r}\bigg)^{s/p}\bigg(\frac{\mu(B(x,\lambda r)}{\mu(\sigma B_0)}\bigg)^{\!\!1/p}
\nonumber\\
&\leq C_H\,\lambda^{-1}\bigg(\frac{R}{r}\bigg)^{s/p}\bigg(\frac{\mu(B)}{\mu(B_0)}\bigg)^{\!\!1/p},
\nonumber
\end{align}
from which \eqref{measboundthm-tt2} follows with 
$\kappa:=(\lambda/C_H)^p\in(0,\infty)$. This finishes the proof
of the reverse implication and, in turn, the proof of the
theorem.
\end{proof}

\end{document}